\documentclass[12pt,reqno]{amsart}

\usepackage{amssymb,bm}

\usepackage{enumerate}
\usepackage{tikz}
\usepackage[margin=0.8in]{geometry}
\usepackage[colorlinks,pdfstartview=FitH]{hyperref}

\hypersetup{
    colorlinks,
    citecolor=blue,
    filecolor=black,
    linkcolor=red,
    urlcolor=black
}

\usepackage{graphicx}

\usepackage{comment}
\usepackage{xcolor}
\usepackage{url}
\usepackage{subcaption}
\usepackage[noadjust]{cite}
\usepackage{mathtools}
\usepackage{extarrows}

\theoremstyle{plain}
\newtheorem{theorem}{Theorem}[section]
\newtheorem{lemma}[theorem]{Lemma}
\newtheorem{proposition}[theorem]{Proposition}
\newtheorem{corollary}[theorem]{Corollary}
\newtheorem{conjecture}[theorem]{Conjecture}
\newtheorem{remark}{Remark}[section]
\newtheorem{example}{Example}[section]
\newtheorem{question}[theorem]{Question}

\numberwithin{equation}{section}

\newcommand{\N}{\mathbb{N}}

\newcommand{\R}{\mathbb{R}}

\newcommand*{\hd}{\dim_{\mathcal{H}}}
\newcommand*{\sd}{\dim_{\mathcal{S}}}

\newcommand*{\bd}{\dim_{\mathcal{B}}}

\begin{document}
\title[Hausdorff dimension of random complex series ]{Hausdorff dimension of images and graphs of some\\ random complex series}
\author{Chun-Kit Lai}
\address{Department of Mathematics, San Francisco State University, San Francisco, CA 94132, USA}
\email{cklai@sfsu.edu}
\author{Ka-Sing Lau}

\author{Peng-Fei Zhang}
\address{Department of Mathematics, Southwest Jiaotong University, Chengdu 611756, China.}
\email{pfzhang@link.cuhk.edu.hk}

\begin{abstract}
Let  $\{X_n= e^{2\pi i \theta_n}\}$ be a sequence  of  Steinhaus random variables, where $\theta_n$ are independent and uniformly distributed on $[0,1]$.  We  compute the almost sure  Hausdorff dimension of the images and graphs of the random complex  series $S(x)=\sum_{n=1}^{\infty}a_n X_n\phi_n(\lambda_nx)$, where $\lambda_n$ is  an increasing sequence with $\sup_n\lambda_{n+1}/\lambda_n<\infty$ and $\phi_n$ satisfies some uniform Lipschitz and boundedness conditions. This class of series includes the famous Weierstrass and Riemann functions as well as others appeared in literature. These results help predict the exact values of the deterministic cases.
\end{abstract}

\maketitle

\section{Introduction}

\subsection{A historical account.} The goal of this paper\footnote{The deterministic version of the problem was first studied by the second-named author, Ka-Sing Lau, when he was investigating the Cantor boundary behavior of the Weierstrass functions. Lau brought the other two authors into the study, for both random and deterministic images. Unfortunately, Lau was deceased in 2021 with the work unfinished. We decide to include all possible generalizations of the random situation and make it public this year.} is to study the fractal dimensions of a randomized model of the
  Weierstrass function $W_{\beta, \lambda}$ and the general Riemann function $R_{a,b}$ as images in the complex planes:
\[
W_{\beta, \lambda}(x)=\sum_{n=1}^{\infty}\lambda^{-\beta n}e^{2\pi i\lambda^nx}\quad (\lambda>1, 0<\beta\leq 1),\qquad R_{a, b}(x)=\sum_{n=1}^{\infty}\frac{e^{2\pi in^{a} x}}{n^{b}}\quad (a>0, b>1).
\]
Readers can refer to Figures \ref{fig1} and \ref{fig2} for some of the intriguing pictures.
Weierstrass and Riemann functions have a long history undoubtedly. In 1872, Weierstrass \cite{Weierstrass} credited Riemann (in or before 1861) with recognizing that the series
\[
\sum_{n=1}^{\infty}\frac{\sin\pi n^2 x}{n^2}
\]
represents a continuous nowhere-differentiable function, though Riemann left no proof of its differentiability.  Weierstrass did not prove it either and instead presented his own celebrated example of nowhere differentiable continuous function
\[
\sum_{n=1}^{\infty}a^n\cos(b^n\pi x),\quad 0<a<1,\;\; ab>1+\frac{3\pi}{2}.
\]
In 1916, Hardy \cite{Hardy}  refined the non-differentiability condition of the Weierstrass function into the now widely recognized strict threshold: $0<a<1$ and $ab\geq 1$; he also proved that the Riemann function is non-differentiable at all irrational points and at some rational points.  It was not until 1970 that Gerver \cite{Ger} obtained a complete characterization of its differentiability, showing that the Riemann function is differentiable only at those rational numbers expressible as a quotient of two odd integers, and its derivative  is $-1/2$. For general $a,b$, it is a standard proof in real analysis that $R_{a, b}$ is differentiable when $b>a+1$. It follows from   \cite{Johnsen}  that  $R_{a, b}(x)$ is nowhere differentiable for  $b\leq a-1$.
Luther  \cite{Luther} proved that the imaginary part of the deterministic  function $R_{2, b}$ is nowhere differentiable for $1 \leq b \leq 3/2$.

Both the Weierstrass and Riemann functions were real-valued trigonometric series in the classical study. Considering it as a complex series has also been studied for more than half a century.

\begin{figure}[ht]
\begin{subfigure}{0.32\textwidth}
\centering\includegraphics[height=3.5cm]{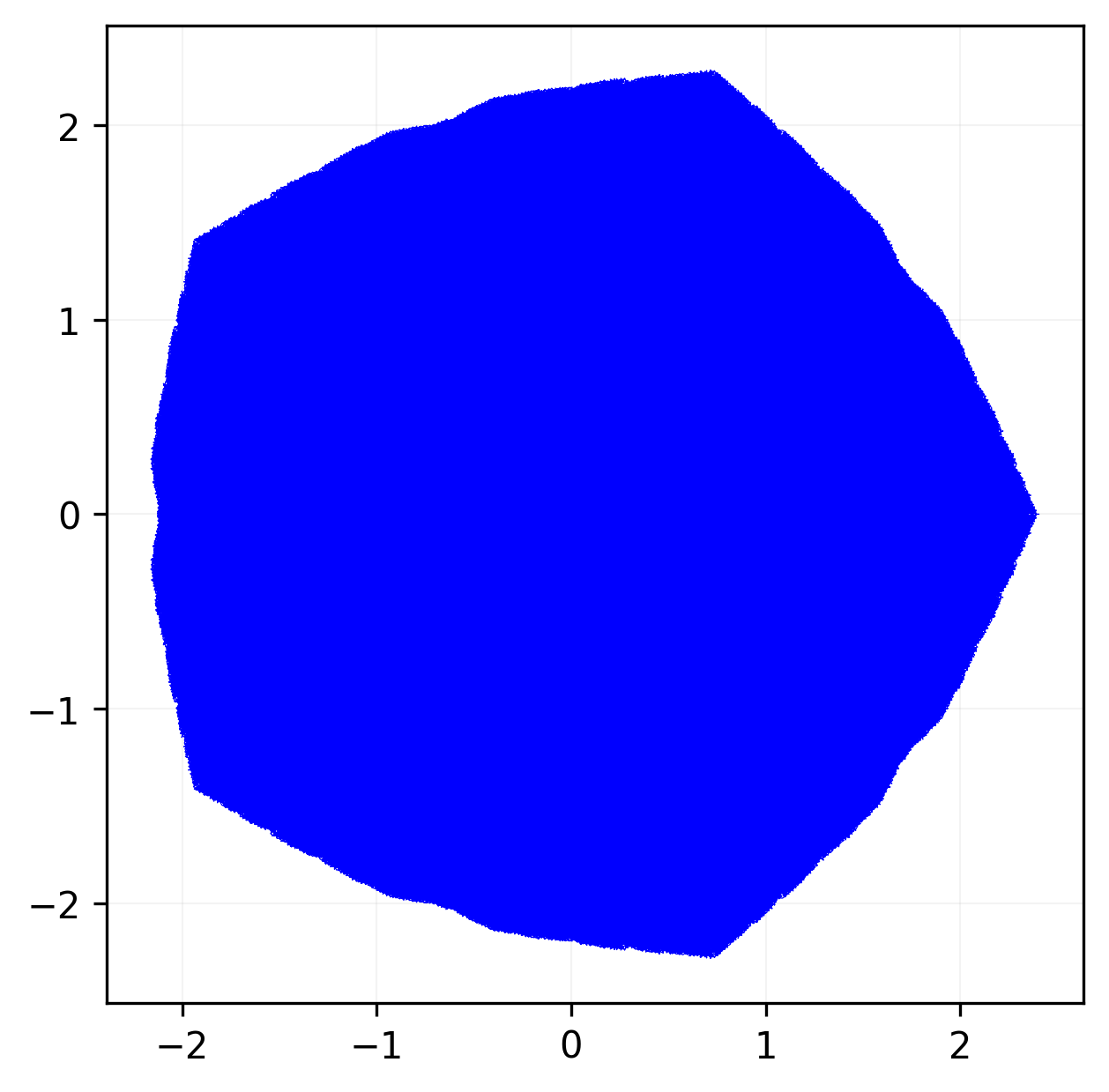}
\caption*{$\lambda=6,\beta=0.3$}
\end{subfigure}
\hfill%
\begin{subfigure}{0.32\textwidth}
\centering\includegraphics[height=3.5cm]{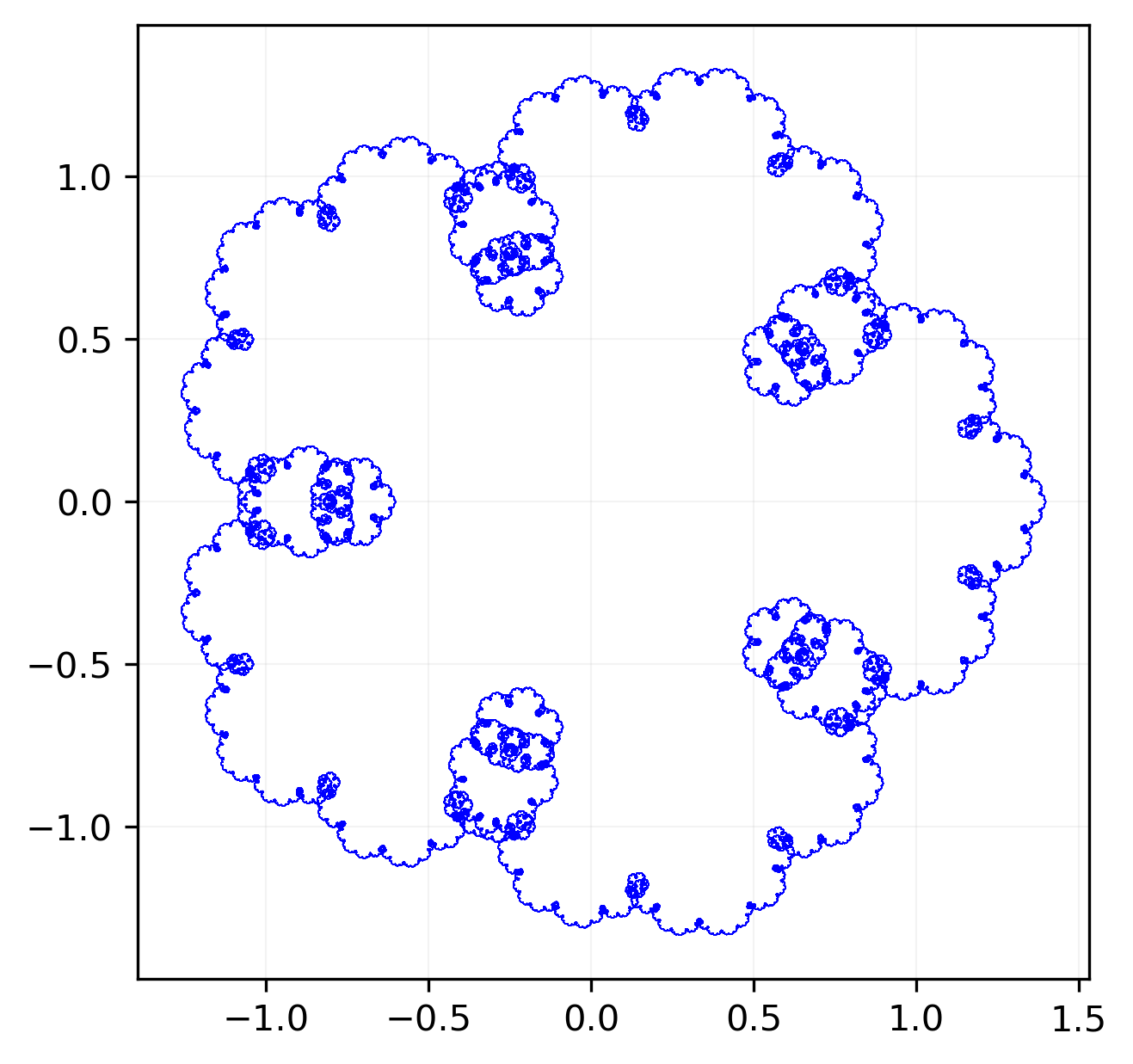}
\caption*{$\lambda=6,\beta=0.7$}
\end{subfigure}
\hfill%
\begin{subfigure}{0.32\textwidth}
\centering\includegraphics[height=3.5cm]{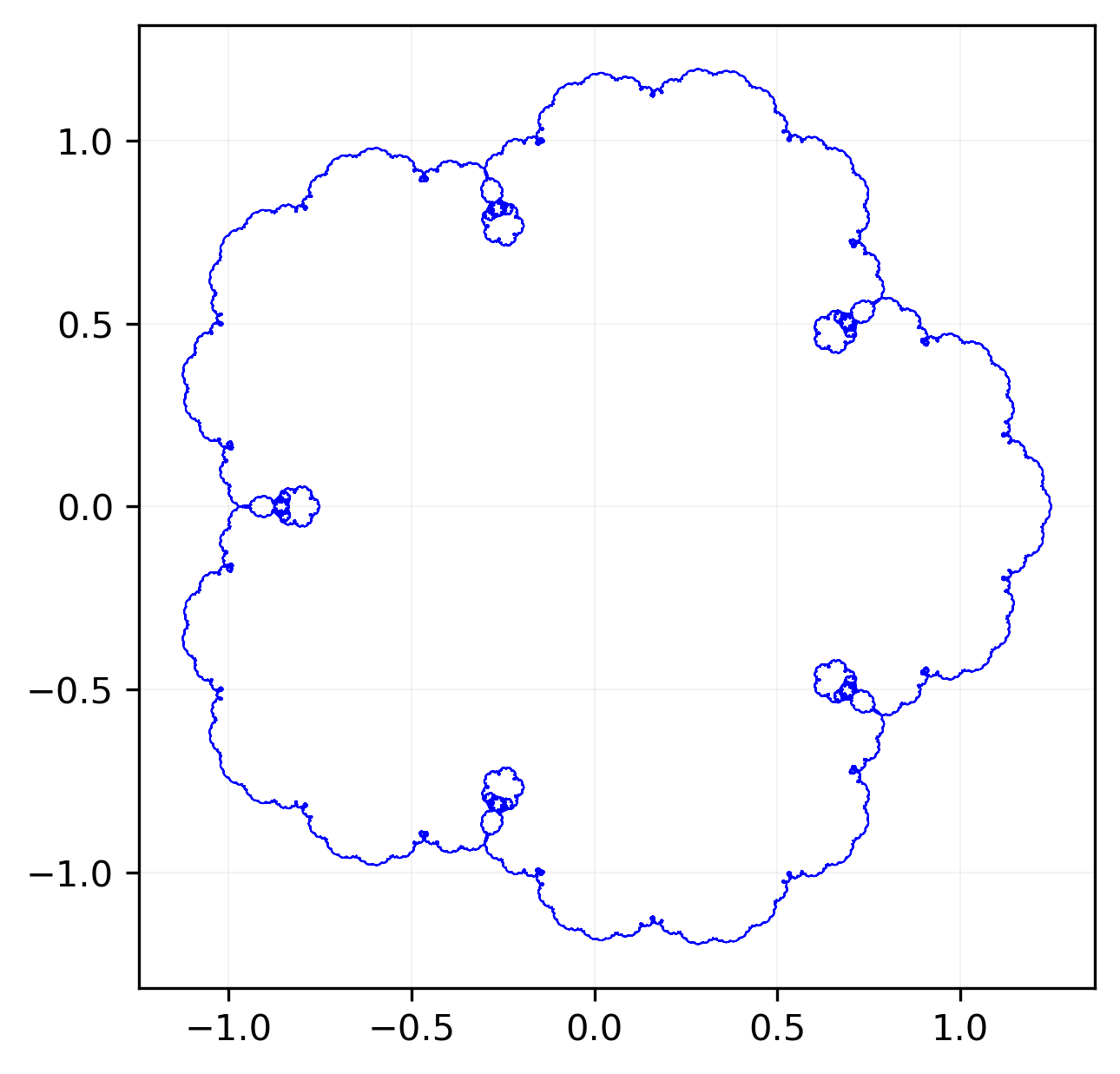}
\caption*{$\lambda=6,\beta=0.9$}
\end{subfigure}

\begin{subfigure}{0.32\textwidth}
\centering\includegraphics[height=3.5cm]{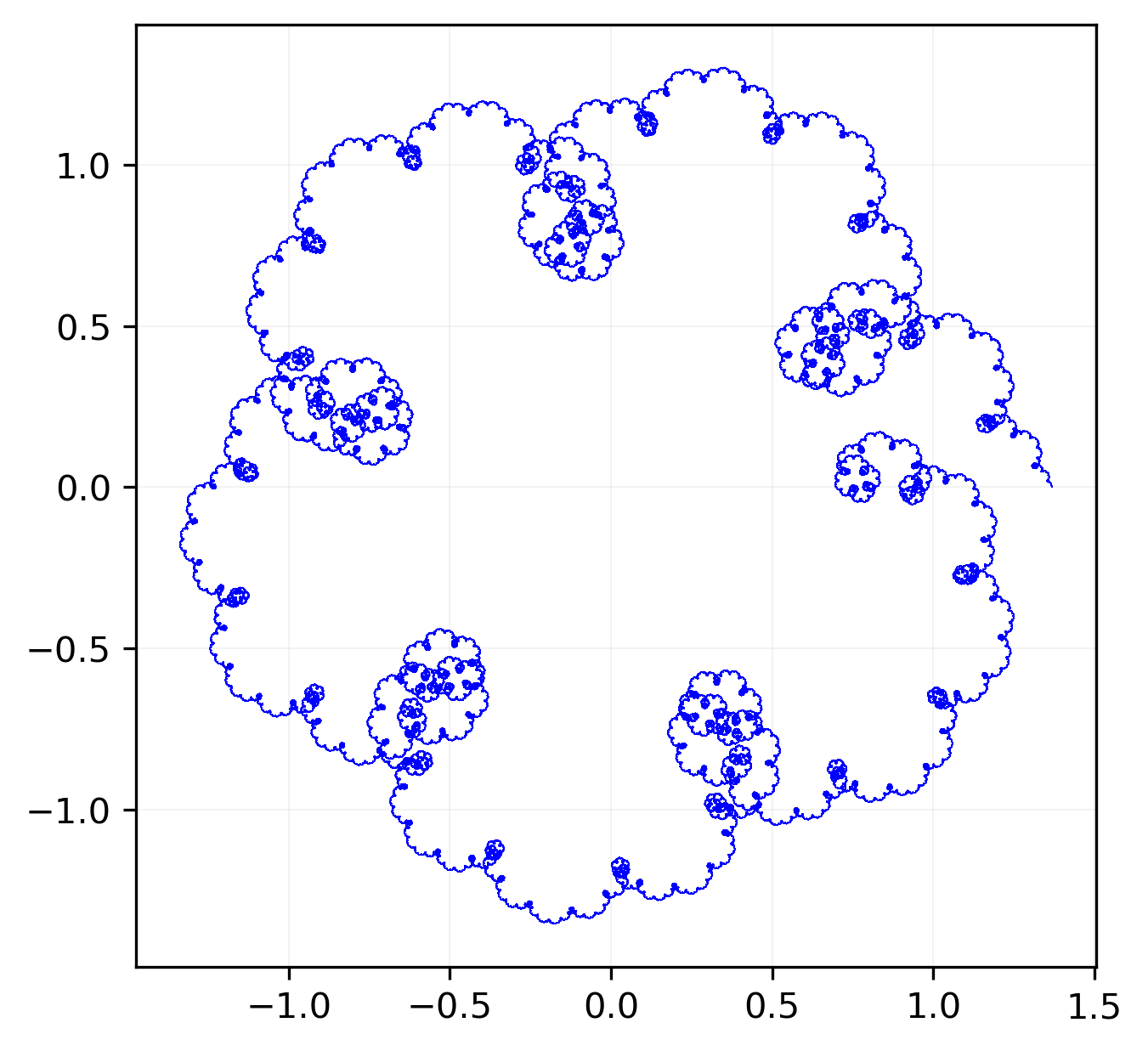}
\caption*{$\lambda=6.5,\beta=0.7$}
\end{subfigure}
\hfill%
\begin{subfigure}{0.32\textwidth}
\centering\includegraphics[height=3.5cm]{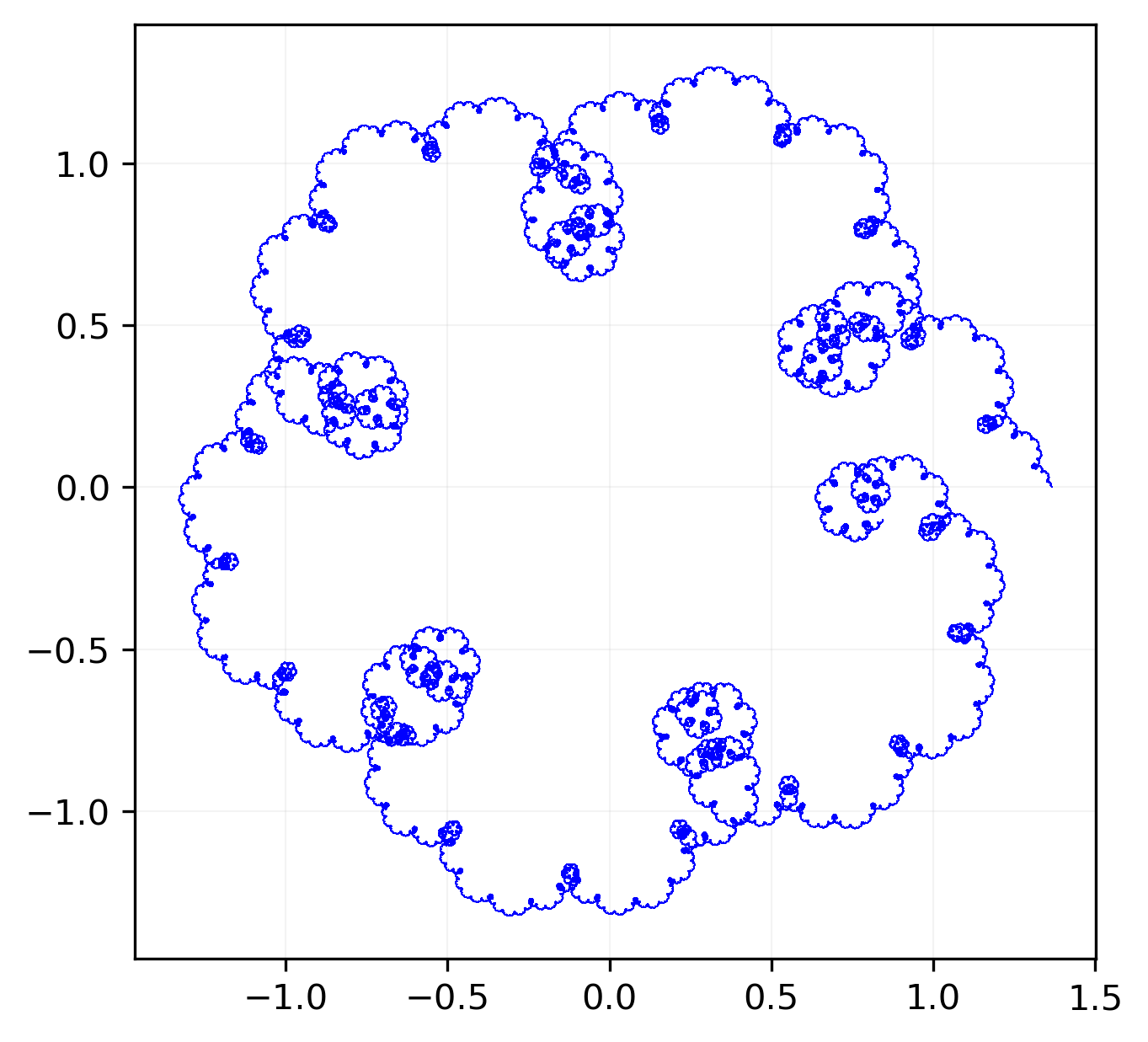}
\caption*{$\lambda=\sqrt{43},\beta=0.7$}
\end{subfigure}
\hfill%
\begin{subfigure}{0.32\textwidth}
\centering\includegraphics[height=3.5cm]{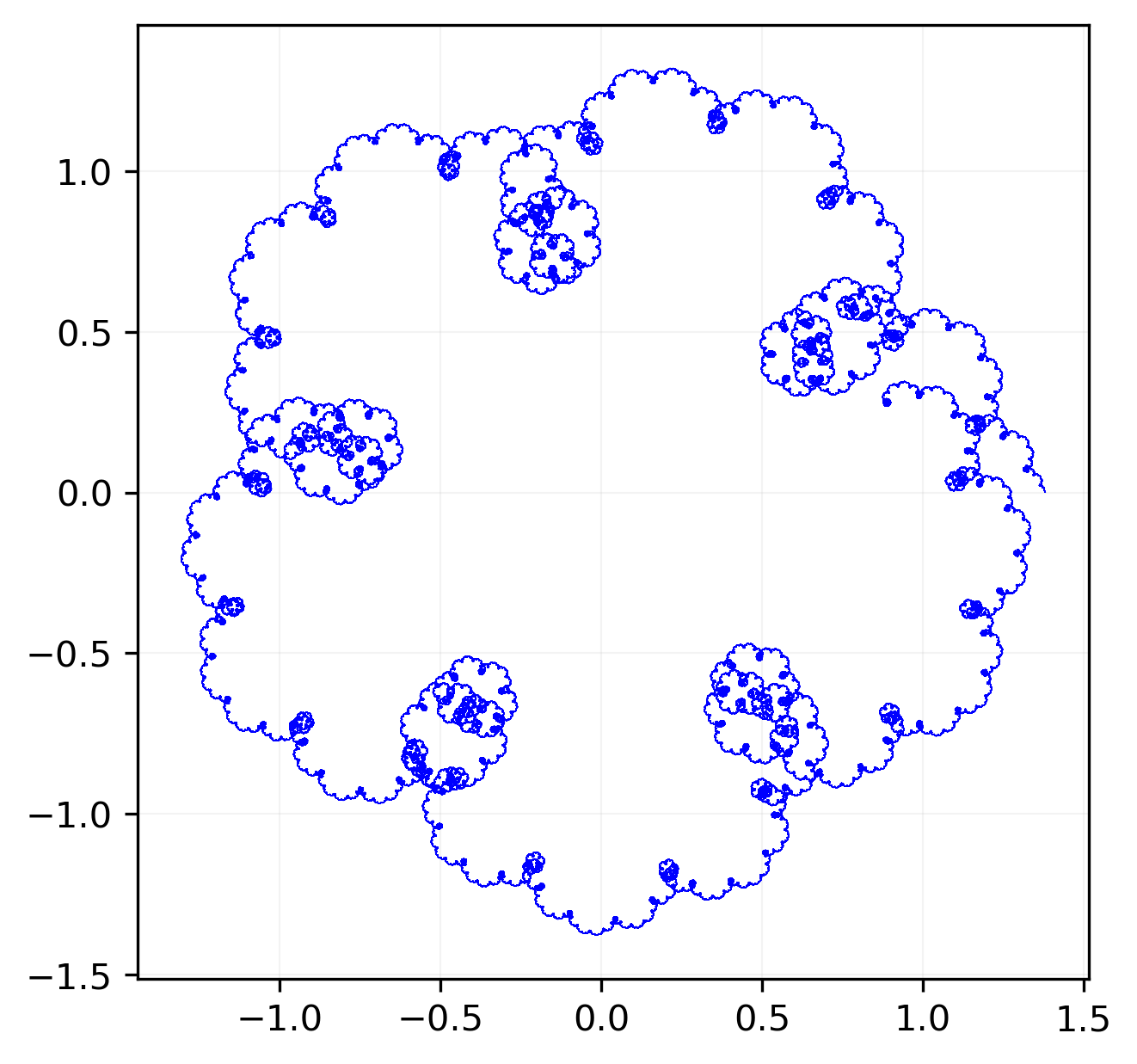}
\caption*{$\lambda=2\pi,\beta=0.7$}
\end{subfigure}

\begin{subfigure}{0.32\textwidth}
\centering\includegraphics[height=3.5cm]{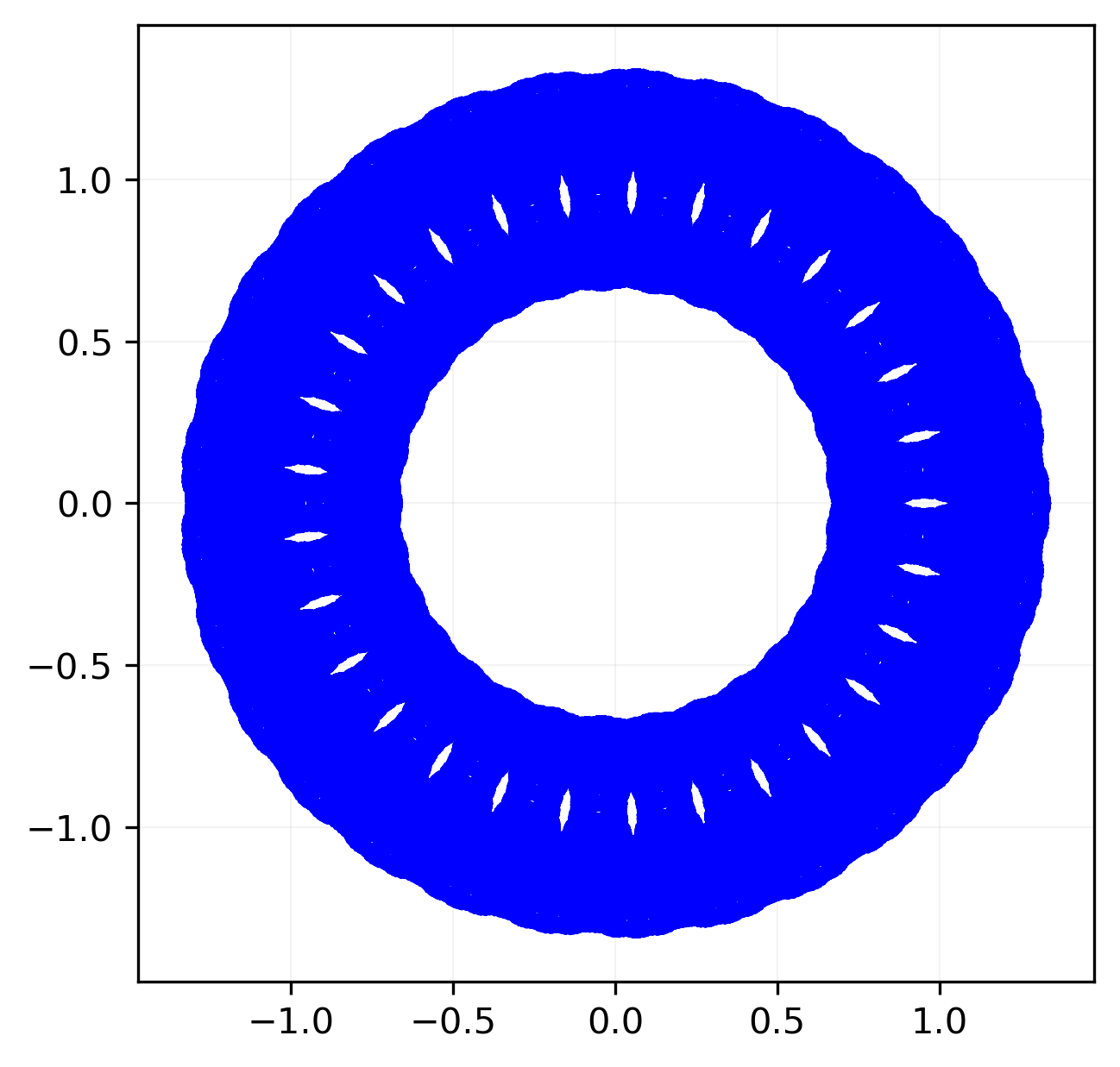}
\caption*{$\lambda=30,\beta=0.4$}
\end{subfigure}
\hfill%
\begin{subfigure}{0.32\textwidth}
\centering\includegraphics[height=3.5cm]{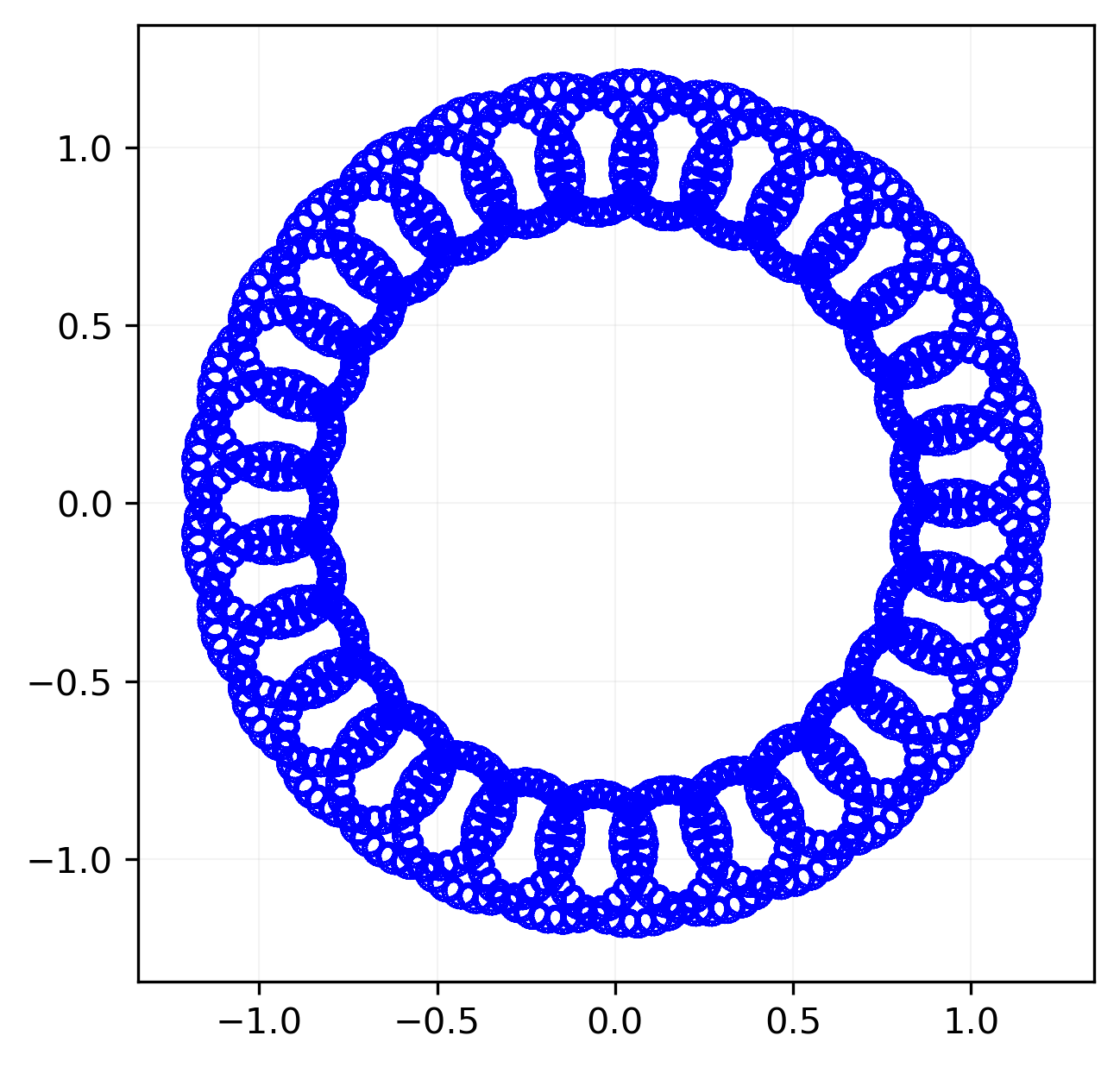}
\caption*{$\lambda=30,\beta=0.5$}
\end{subfigure}
s\hfill%
\begin{subfigure}{0.32\textwidth}
\centering\includegraphics[height=3.5cm]{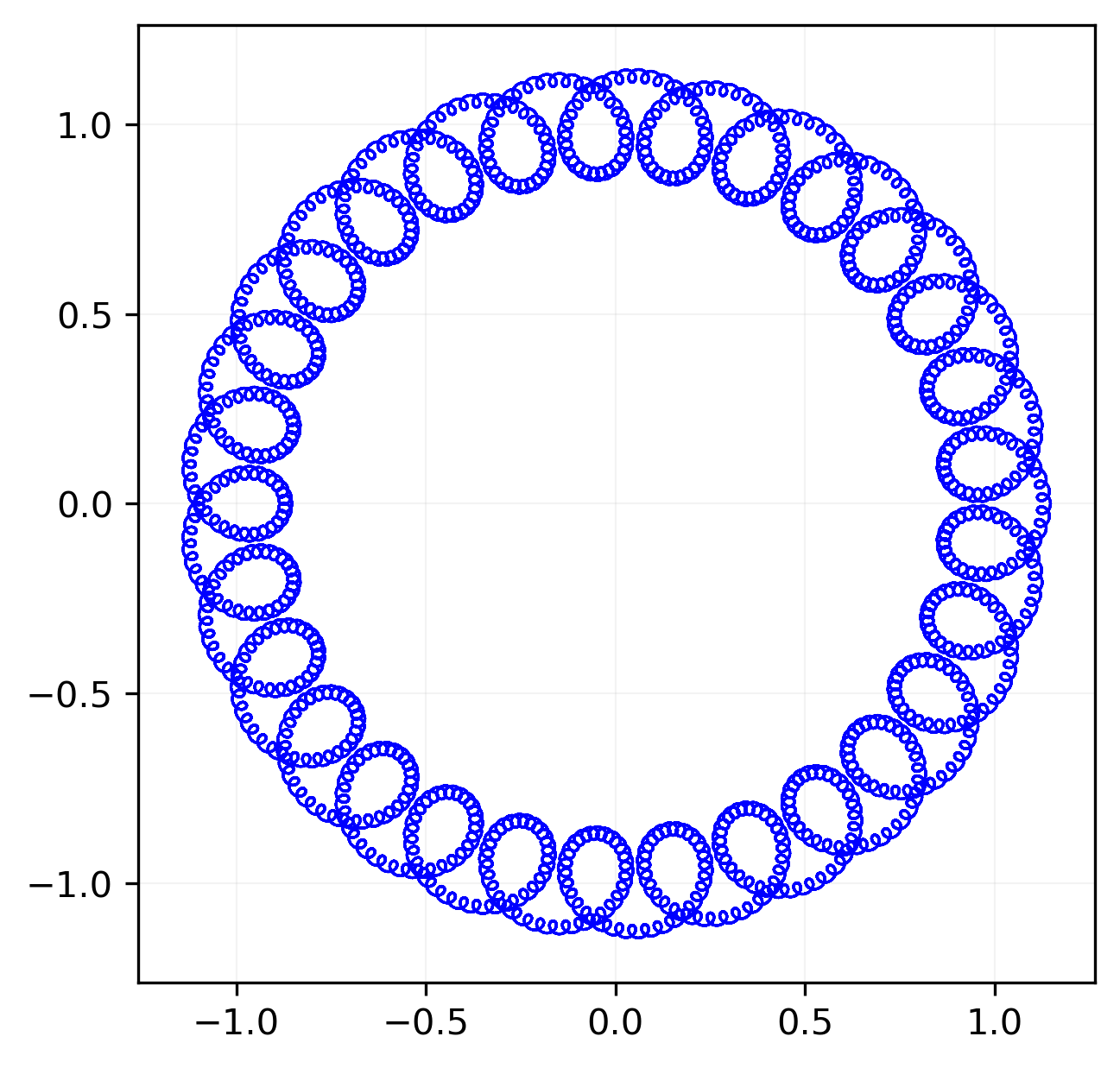}
\caption*{$\lambda=30,\beta=0.6$}
\end{subfigure}
\caption{Curves $W_{\beta, \lambda}([0,1])$ for different $\beta$ and $\lambda$}\label{fig1}
\end{figure}

\begin{figure}[ht]
\begin{subfigure}{0.32\textwidth}
\centering\includegraphics[height=3.5cm]{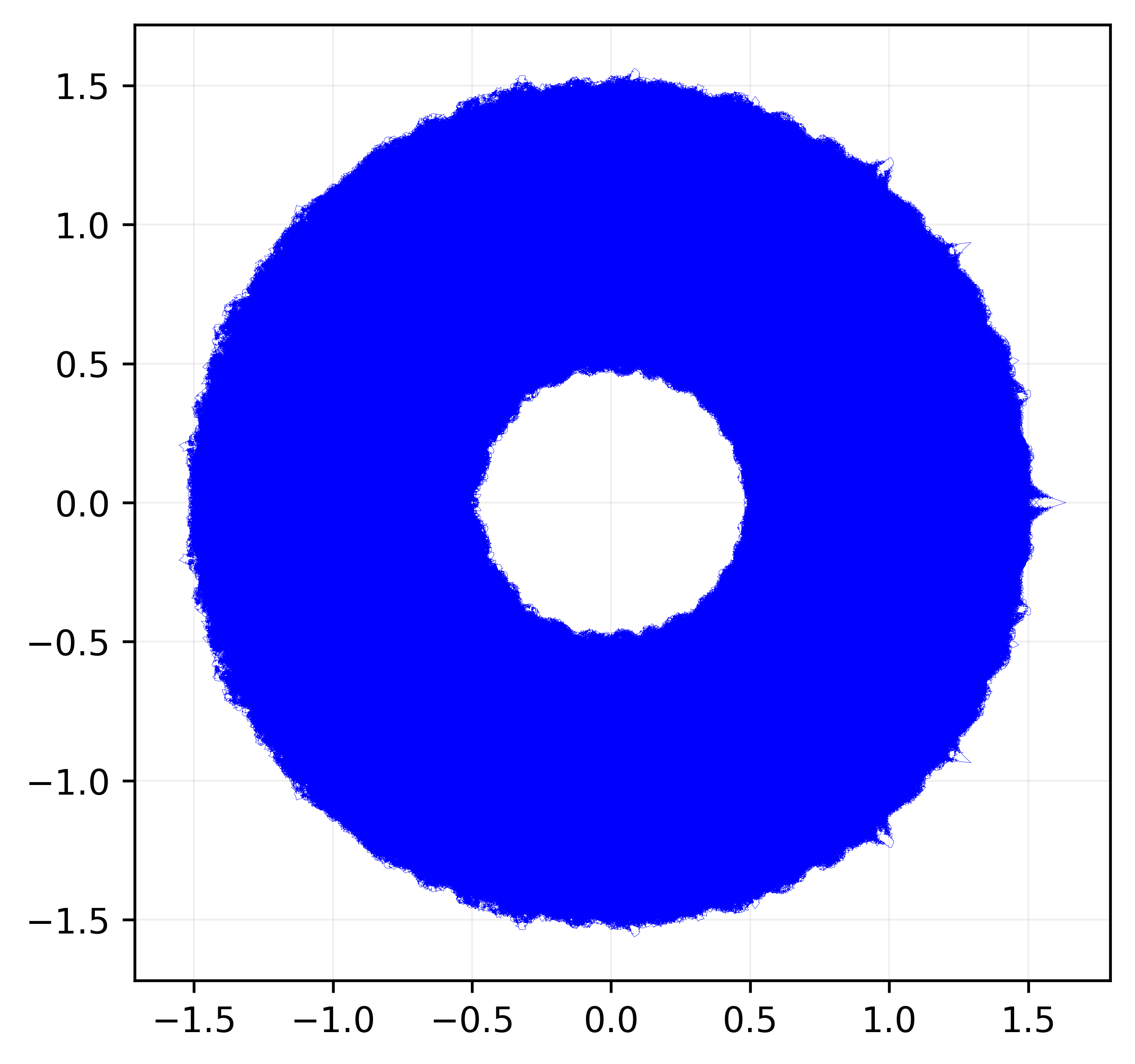}
\caption*{$a=6,b=2,\tau=0.25$}
\end{subfigure}
\hfill%
\begin{subfigure}{0.32\textwidth}
\centering\includegraphics[height=3.5cm]{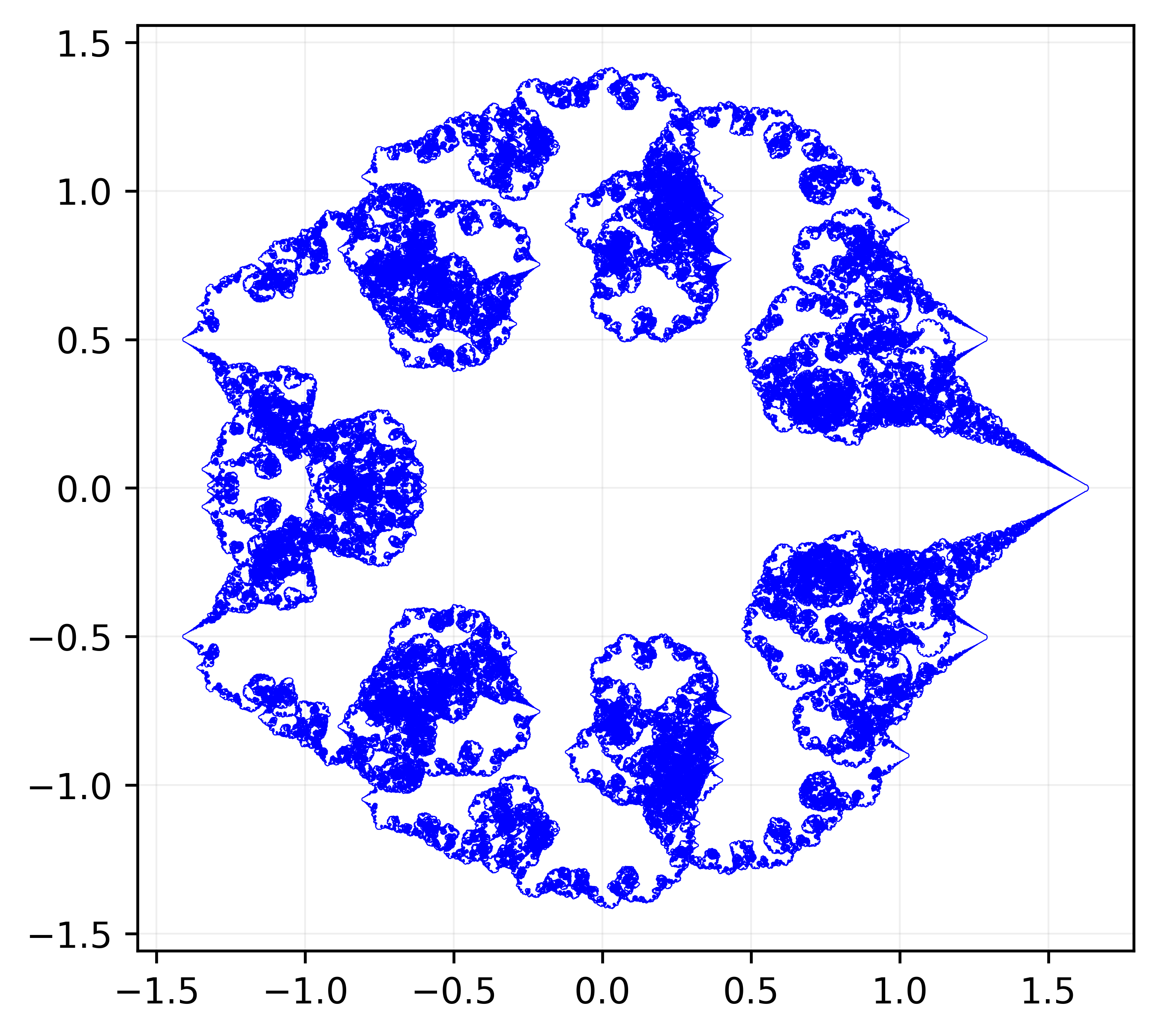}
\caption*{$a=3,b=2,\tau=0.5$}
\end{subfigure}
\hfill%
\begin{subfigure}{0.32\textwidth}
\centering\includegraphics[height=3.5cm]{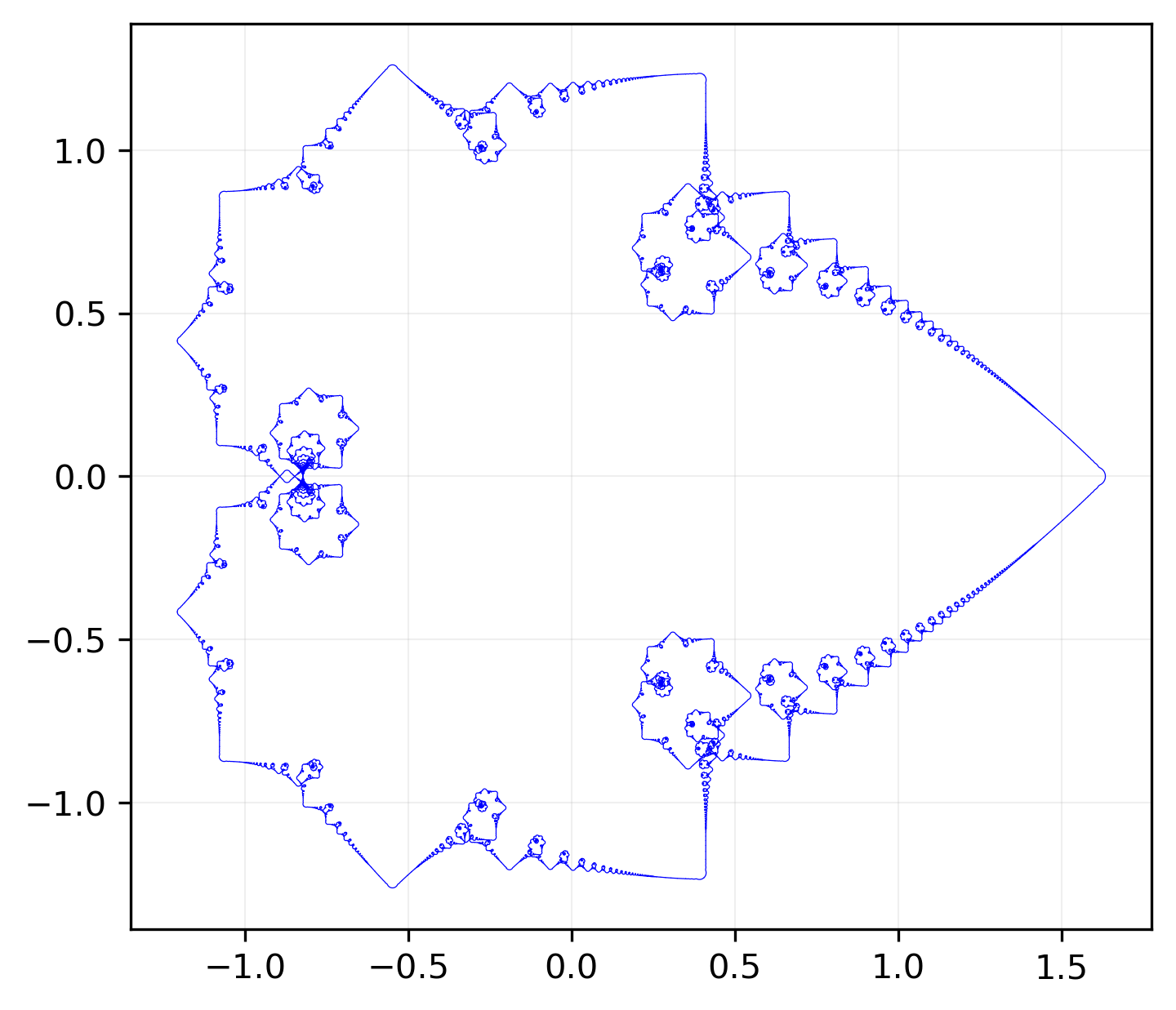}
\caption*{$a=2,b=2,\tau=0.75$}
\end{subfigure}
\caption{Curves $R_{a,b}([0, 1])$ for different $a$ and $b$.}\label{fig2}
\end{figure}

One notable result was due to Salem and Zygmund \cite{SZ}, who discovered the space-filling behavior of lacunary series. They proved that,  under certain conditions, the Hadamard gap series $\sum_{k=1}^{\infty}a_k z^{n_k}$ ($n_{k+1}/n_k\geq \lambda>1$) on $\partial\mathbb{D}$ is space-filling. A {\bf space-filling curve} is a continuous curve from $\R$ to $\R^d$ ($d>1$) whose image contains an open set. Classic examples include the Peano curve and the Hilbert curve, but Salem and Zygmund's work was the first to provide an explicit {\it analytic} construction of a space-filling curve. Kahane, Weiss, and Weiss \cite{KWW}  improved the result of Salem and Zygmund and proved that there exists a Cantor set on which the Hadamard gap series is space-filling. Schaeffer \cite{Schaeffer} proved that, under certain conditions, the image of $\partial\mathbb{D}$ under the Hadamard gap series  coincides with the image of the closed unit disk $\overline{\mathbb{D}}$. Their results also hold for the complex Weierstrass function $W_{\beta, \lambda}$ when $\beta$ is sufficiently small; Bara\'{n}ski \cite{Bar} also obtained this space-filling behavior.  However, their work did not specify how small $\beta$ should be. Belov \cite{Belov} answered this question and proved that if $\lambda>\lambda_0$ and $\beta\in(0, 1/2)$  satisfy $\lambda_0^{1-2\beta}=\pi+\lambda_0^{-\beta}$, then $W_{\beta, \lambda}$ is space-filling on every sufficiently long arc on $\partial\mathbb{D}$.

\subsection{Hausdorff dimension.} Apart from the topological properties,  Hausdorff dimension of the sets related to the Weierstrass functions has well-received attentions since the birth of the modern fractal geometry after Mandelbrot. It has long been conjectured that $2-\beta$ is the Hausdorff dimension of the graph of the real-valued Weierstrass functions $\mathrm{Re}\, W_{\beta, \lambda}$ and $\mathrm{Im}\, W_{\beta, \lambda}$. This is the natural upper bound obtained from the optimal H\"{o}lder exponents of the Weierstrass functions \cite{Fal}.   Hunt \cite{Hunt} confirmed this dimension probabilistically via the random phases.  Bara\'{n}ski,  B\'{a}r\'{a}ny, and Romanowska \cite{BBR} confirmed the conjecture  for integer $\lambda $ and  certain small values of $ \beta$. A further breakthrough was achieved by Shen \cite{Shen}, who completely resolved the problem for all $\beta\in (0,1)$ and integer $\lambda$. Subsequently, in joint work with Ren \cite{RS2021,RS2024}, the theory was extended to higher-dimensional Weierstrass-type series. More precisely, they showed that the Hausdorff dimension of the graph of the deterministic high-dimensional Weierstrass-type function with $\phi$ real-analytic and periodic,
\[
W^{\phi}_{\beta, \lambda}(x)=\sum_{n=1}^{\infty}\lambda^{-\beta n}\phi(\lambda^nx),
\]
under the assumption that $\lambda\geq 2$ is an integer, exhibits a dichotomy: it is either $C^k$  or not Lipschitz. The dimension of its graph (over $[0, 1]$) is given by
\begin{equation}\label{eq1.1}
\bd \mathrm{graph}\, W^{\phi}_{\beta, \lambda}=\hd \mathrm{graph}\, W^{\phi}_{\beta, \lambda}=\min\Bigl\{\frac{1}{\beta}, 1+(d-q)(1-\beta)\Bigr\},
\end{equation}
where the integer $q\in[0, d]$ measures the maximal number of smooth directions of $W^{\phi}_{\beta, \lambda}$. For example, in $\mathbb{R}^2$, if $\phi(x)=e^{2\pi ix}$, then $q=0$; if  $\phi(x)=e^{2\pi ix}-\lambda^{-\beta} e^{2\pi i\lambda x}$, then $q=2$ and
\begin{equation}\label{eq1.2}
W^{\phi}_{\beta, \lambda}=\sum_{n=1}^{\infty}\lambda^{-\beta n}(e^{2\pi i \lambda^n x}-\lambda^{-\beta}e^{2\pi i\lambda^{n+1}x})=\lambda^{-\beta}e^{2\pi  i\lambda x}.
\end{equation}
In this case, $W^{\phi}_{\beta, \lambda}$ is smooth. The key tool in their proofs was to express the graph of the Weierstrass functions as an repeller for a expanding dynamical system, which allowed us to use the Ledrappier-Young's dimension theory \cite{Led,LY}.

To the best of our knowledge, the Hausdorff dimension of the graph of the real-valued Riemann functions Re $R_{a,b}$ and Im $R_{a,b}$ is still unknown.  Chamizo, C\'{o}rdoba, and Ubis  computed the box dimension of the real-valued Riemann functions \cite{CC1999,CU2007}. In particular, $\bd \mathrm{graph}\, \mathrm{Re}\,R_{2,2} =\bd \mathrm{graph}\, \mathrm{Im}\,R_{2,2}= 5/4$.

There are unfortunately much less results about the Hausdorff dimension of the image of the complex series, both in the Weierstrass and Riemann series, beyond the space-filling situation.  We are unaware of any dimension results when $\beta\geq 1/2$ (which will not be space-filling) for the Weierstrass functions, and the dynamical system approach does not seem to directly carry over.

 Riemann functions  was originally introduced as a purely analytic example exhibiting pathological regularity behavior.  Surprisingly,  it was discovered in \cite{DV} that the complex  Riemann function $R_{2,2}$ exhibits an unexpected geometric connection with vortex filament dynamics.  A  generalization of the Riemann function
\begin{equation}\label{eq1.3}
\phi(t)=\sum_{n\in\mathbb{Z}}\frac{e^{-4\pi^2 i n^2t}-1}{-4\pi^2 n^2}=-\frac{1}{2\pi^2}R_{2,2}(-2\pi t)+it+\frac{1}{12}
\end{equation}
accurately describes the trajectories of the corners of polygonal vortex filaments governed by the vortex filament equation. Eceizabarrena \cite{Ece} obtained an upper estimate of the Hausdorff dimension its image:
\begin{equation}\label{eq1.4}
    1\leq\hd\phi(\mathbb{R})\leq \frac{4}{3}.
\end{equation}

\subsection{Main Results.} Due to the lack of deterministic dimension results for the image of complex series, our main goal is to introduce a random model similar to Hunt \cite{Hunt}, where the randomness is introduced through random phases, to help us predict the correct value of the Hausdorff dimension.

It turns out that the random model can be studied in a much more general context.  We will replace the exponential function $e^{2\pi ix}$ with  a sequence of functions   $\phi_n:\mathbb{R}\to\mathbb{C}$  that is   {\bf uniformly bounded and locally uniformly bi-Lipschitz }, which means that there exist constants  $L \geq 1$ and $\delta>0$ such that $\sup_n\lVert\phi_n\rVert<\infty$ and
\begin{equation}\label{eq1.5}
L^{-1}|x-y|\leq |\phi_n(x)-\phi_n(y)|\leq L |x-y|,\quad |x-y|\leq \delta,\quad\forall \, n\in \N.
\end{equation}
 Let  $\{a_n\}\subset\mathbb{C}\setminus\{0\}$ with $\sum_{n=1}^{\infty}|a_n|<\infty$, and let $\{\lambda_n\}\subset\mathbb{R}$ be  an increasing  sequence satisfying
\begin{equation}\label{eq1.6}
\lambda_1=1,\quad \sup_n\frac{\lambda_{n+1}}{\lambda_n}:=q<\infty.
\end{equation}
Similar to Kahane \cite{Kahane},  we define the following quantities:
\begin{equation}\label{eq1.7}
s_k=\Bigl(\sum_{q^k\leq \lambda_n<q^{k+1}}|a_n|^2\Bigr)^{\frac{1}{2}}, \qquad \sigma=\liminf_{k\to\infty}\frac{-\log s_k}{k\log q},\qquad \tau=\limsup_{k\to\infty}\frac{-\log s_k}{k\log q}.
\end{equation}
The central object of this paper is the random  series
\begin{equation}\label{eq1.8}
S(x)=S_{\Theta}(x)=\sum_{n=1}^{\infty} a_nX_n\phi_n(\lambda_nx),\quad x\in \mathbb{R},
\end{equation}
where $\Theta=\{\theta_n\}$ is a sequence of independent uniform distributions on $[0, 1]$ ($\theta_n\overset{\mathrm{iid}}{\sim} U[0,1]$). The sequence $\{X_n=e^{2\pi i\theta_n}\}$ is also known as the {\bf Steinhaus sequence}. The following is our main result in this paper.

\begin{theorem}\label{thm1.1}
Suppose that $0<\sigma\leq \tau<\infty$. Then for any Borel set $A\subset\mathbb{R}$, $S(x)$ in (\ref{eq1.8}) has the following property:
\begin{enumerate}[(i)]
    \item almost surely,
\[
\min\Bigl\{2,\; \frac{\hd A}{\tau}\Bigr\}\leq \hd S(A)\leq \min\Bigl\{2,\; \frac{\hd A}{\min\{\sigma, 1\}}\Bigr\}.
\]
\item If $0<\tau<\frac{1}{2}\hd A$, then  $S(A)$ has positive two-dimensional Lebesgue measure a.s.
\item If $0<\tau<\frac{1}{4}\hd A$, then  $S (A)$ has interior a.s.
\end{enumerate}
Moreover,  if $\tau=0$ and  $\hd A>0$,  then $S(A)$ has interior a.s.
\end{theorem}

One can also randomize the coefficients using Gaussian distributions. A special case of this is the  Gaussian Fourier series (see Kahane \cite[Chapter 14]{Kahane}, in which a similar result to (i) was obtained. As Gaussian model changes to size of all the coefficients, Steinhaus model appears to be more natural to reflect on the possible values on the deterministic case.

Let $G_S( A)=\{(x, S(x))\colon x\in A\}\subset \R^3$ be the graph of the random complex series $S(x)$ in \eqref{eq1.8} over $A$. We also obtain the Hausdorff dimension of its graphs.
\begin{theorem}\label{thm1.2}
If $0< \sigma\leq \tau\leq 1$, for any Borel set $A\subset\mathbb{R}$, then almost surely,
\[
\min\Bigl\{\frac{\hd A}{\tau},\; \hd A+2-2\tau\Bigr\}\leq \hd G_S( A)\leq\min\Bigl\{\frac{\hd A}{\sigma},\; \hd A+2-2\sigma\Bigr\}.
\]
Moreover, if $\tau=0$ and $\hd A>0$, then $\hd G_S( A)=2+\hd A$ a.s.
\end{theorem}

We will now use these results to obtain our dimension results for the Weierstrass and Riemann functions. To the best of our knowledge, existing studies (random and deterministic) on the Hausdorff dimension of graphs of lacunary series are confined to the Weierstrass-type series or Hadamard gap series ($\inf_n \lambda_{n+1}/\lambda_n>1$). Our condition \eqref{eq1.6} allows us to  include also  the   Riemann-type function, in which $\lim_{n\to\infty}\lambda_{n+1}/\lambda_n= 1$.

\noindent{\bf (A). Weierstrass functions.} Our randomized Weierstrass function given by the Steinhaus sequence is defined as
\begin{equation}\label{eq1.9}
W_{\beta,\lambda,\Theta}(x)=\sum_{n=1}^{\infty}\lambda^{-\beta n}  e^{2\pi i(\lambda^{n}x+\theta_n)},\quad \beta\in(0, 1],\;\; \lambda>1,
\end{equation}
where $\theta_n\overset{\mathrm{iid}}{\sim} U[0,1]$.  Then a direct calculation of (\ref{eq1.7}) gives us $s_k=\lambda^{-\beta k}$ and $\sigma=\tau=\beta$. This gives us
\begin{corollary}\label{coro1.3}
Let $0< \beta\leq 1$  and $\lambda>1$.    For any Borel set $A\subset\mathbb{R}$, then
\begin{enumerate}[(i)]
\item almost surely,
\begin{equation} \label{eq1.10}
\begin{aligned}
\hd W_{\beta, \lambda,\Theta}(A)&=\min\left\{2,\; \frac{\hd A}{\beta}\right\},\\
\hd G_{W_{\beta, \lambda,\Theta}}( A)&=\min\Bigl\{\frac{\hd A}{\beta},\; \hd A+2-2\beta\Bigr\}.
\end{aligned}
\end{equation}
\item If  $\beta<\frac{1}{2}\hd A$, then $W_{\beta, \lambda,\Theta}(A)$ has positive two-dimensional Lebesgue measure a.s.
\item If $\beta<\frac{1}{4}\hd A$, then $ W_{\beta, \lambda,\Theta}(A)$ has an interior a.s.
\item If $\bd A=\hd A$, then \eqref{eq1.10} also holds for the box-counting dimension.
\end{enumerate}
\end{corollary}
The box dimension is no less than the Hausdorff dimension, and the upper bounds are both determined by the H\"{o}lder exponent. Therefore, (iv) follows immediately.

Another interesting application is towards the function  $\phi(x) \equiv e^{2\pi ix}-\lambda^{-\beta} e^{2\pi i\lambda x}\, (0<\beta<1,\lambda>1)$, which is  locally uniformly bi-Lipschitz. Without any randomization, (\ref{eq1.2}) shows that $W^{\phi}_{\beta,\lambda}$ is a smooth function, so its graph has Hausdorff dimension 1. Its associated random series is given by
\[
W^{\phi}_{\beta,\lambda,\Theta}(x)=\sum_{n=1}^{\infty}\lambda^{-\beta n}X_n\phi(\lambda^nx)
\]
where $X_n = e^{2\pi i \theta_n}$ is the Steinhaus sequence.  Theorem \ref{thm1.2} now shows that the dimension of the graph  of $W^{\phi}_{\beta,\lambda,\Theta}([0, 1])$ is $\min\{\beta^{-1}, 3-2\beta\}$ almost surely.  Compared to the deterministic case, randomization of the phases destroys the smoothness and let it exhibit the fractal properties.  Rotation destroys the smoothness of $W^{\phi}_{\beta, \lambda}$ even with deterministic rotation.  Figures \ref{fig3} and \ref{fig4} show the  curves $W^{\phi}_{\beta,\lambda,\Theta}([0, 1])$ for $\phi(x)=e^{2\pi ix}-\lambda^{-\beta} e^{2\pi i\lambda x}$ and different models $\Theta$.
\begin{figure}[ht]
\begin{subfigure}{0.45\textwidth}
\centering\includegraphics[height=3.5cm]{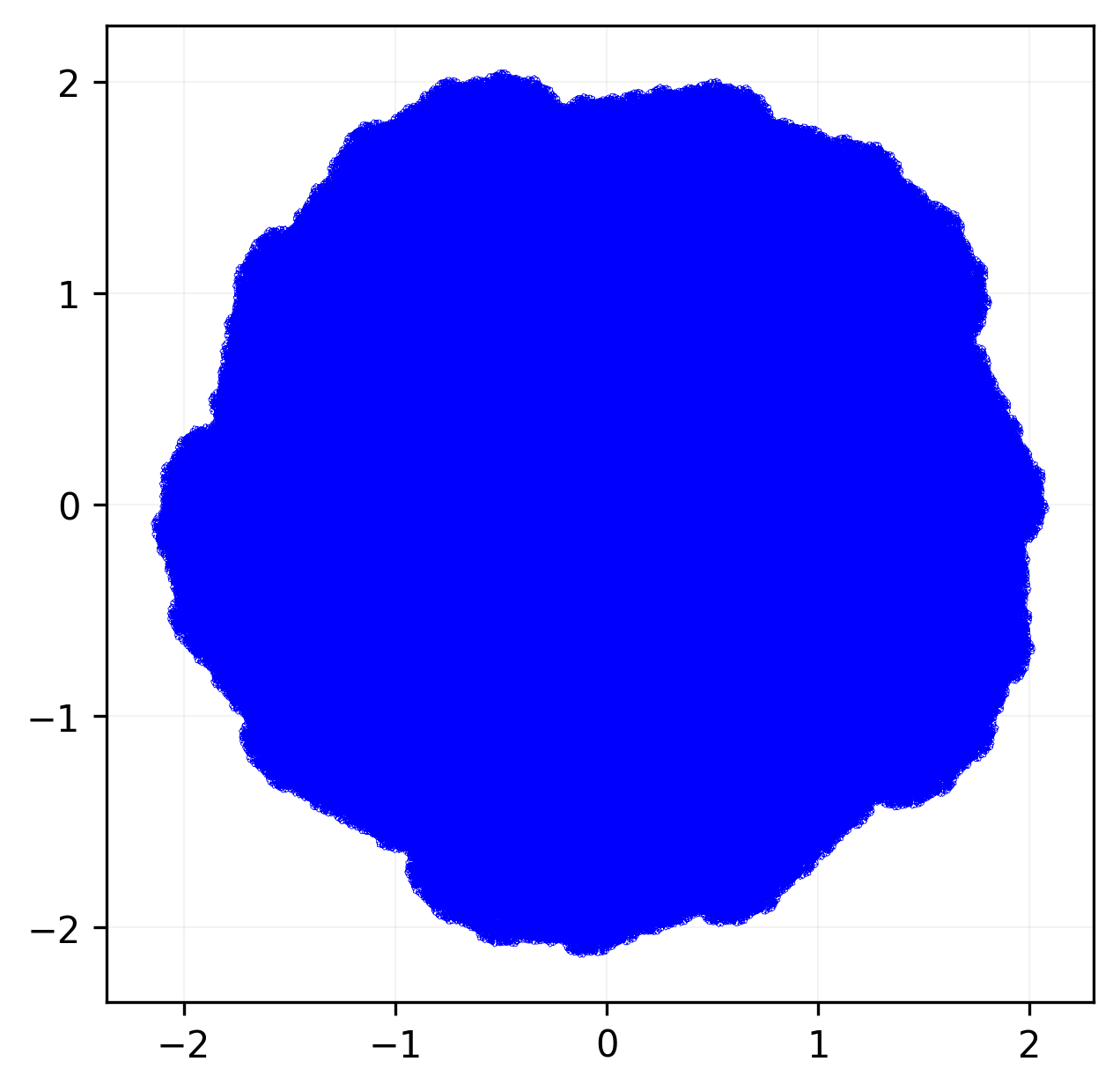}
\caption*{$\lambda=6,\beta=0.3$}
\end{subfigure}
\hfill%
\begin{subfigure}{0.45\textwidth}
\includegraphics[height=3.5cm]{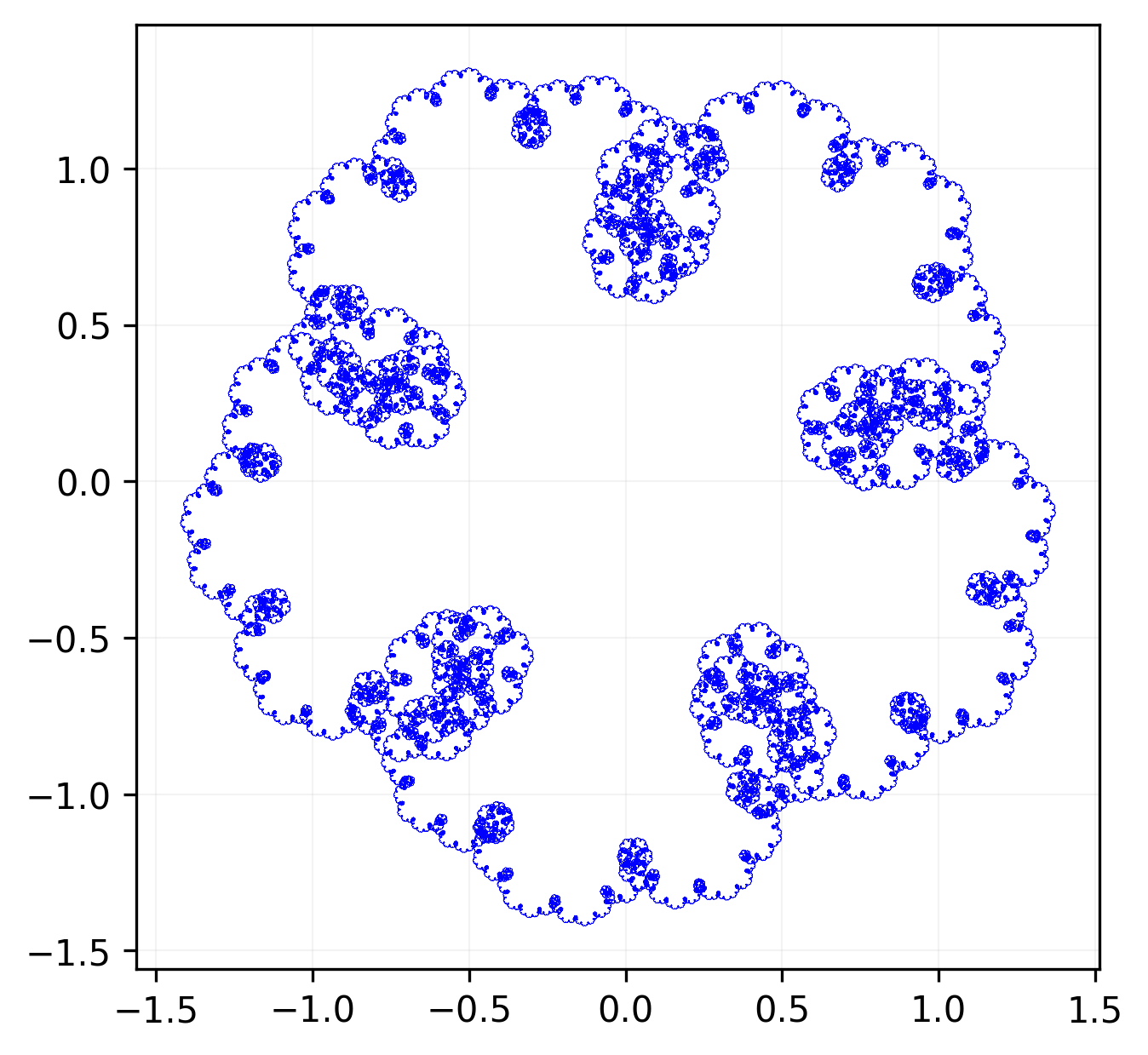}
\centering\caption*{$\lambda=6,\beta=0.6$}
\end{subfigure}
\caption{Curves $W^{\phi}_{\beta,\lambda,\Theta}([0, 1])$  for $\phi(x)=e^{2\pi ix}-\lambda^{-\beta} e^{2\pi i\lambda x}$ and an equidistributed sequence $\theta_n=n\pi\pmod 1$.}\label{fig3}
\end{figure}

\begin{figure}[ht]
\begin{subfigure}{0.45\textwidth}
\centering\includegraphics[height=3.5cm]{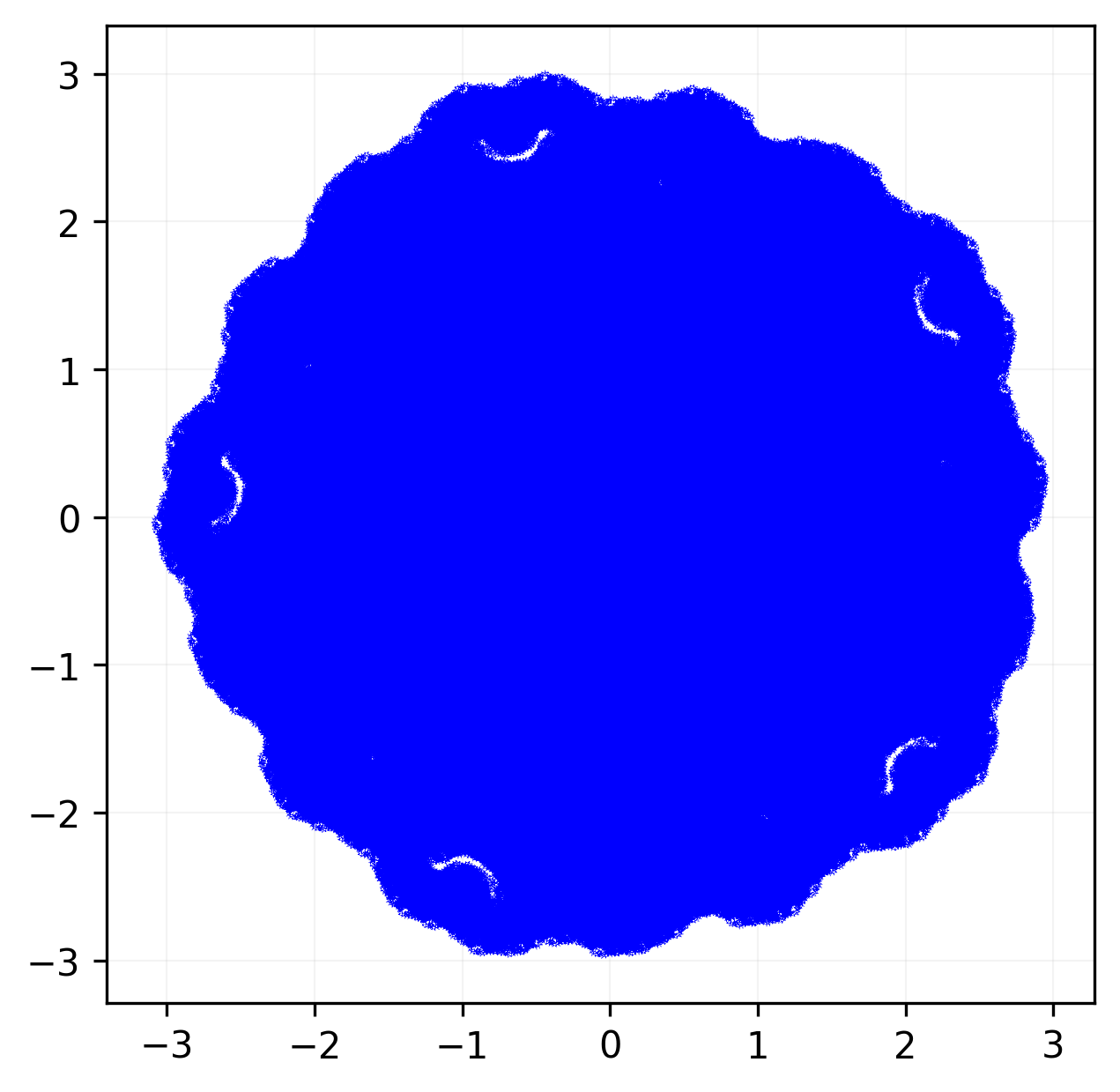}
\caption*{$\lambda=6,\beta=0.3$}
\end{subfigure}
\hfill%
\begin{subfigure}{0.45\textwidth}
\centering\includegraphics[height=3.5cm]{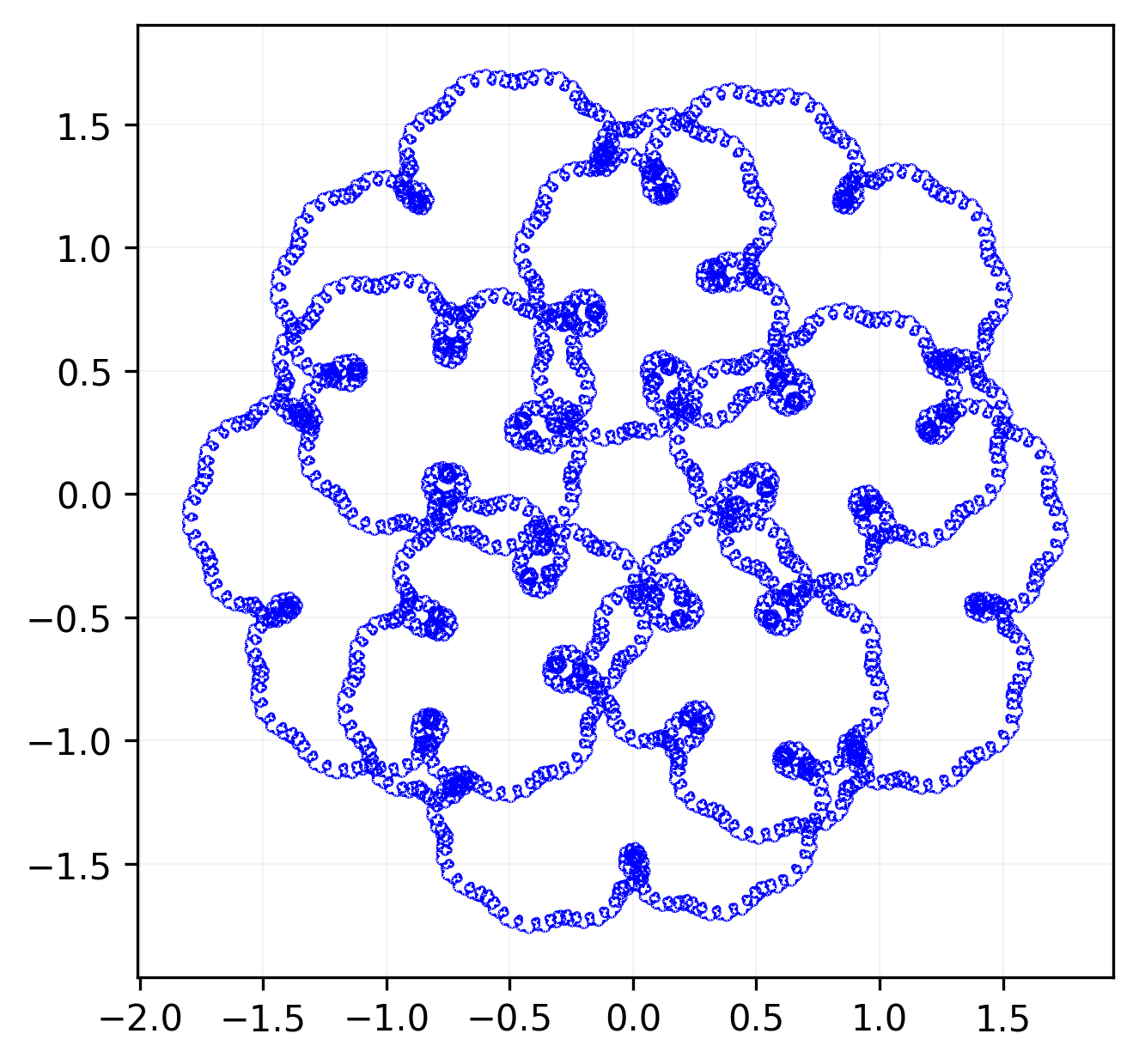}
\caption*{$\lambda=6,\beta=0.6$}
\end{subfigure}
\caption{Sample paths of $W^{\phi}_{\beta,\lambda,\Theta}([0, 1])$ for $\phi(x)=e^{2\pi ix}-\lambda^{-\beta} e^{2\pi i\lambda x}$ and  $\theta_n\overset{\mathrm{iid}}{\sim} U[0,1]$.}\label{fig4}
\end{figure}

\noindent{\bf (B). Riemann Functions.} The random Riemann function is given by
\begin{equation}\label{eq1.11}
R_{a, b, \Theta}(x)=\sum_{n=1}^{\infty}  n^{-b}  e^{2\pi i (n^{a}x+\theta_n)},\quad a>0,\;\; b>1,
\end{equation}
where $\theta_n\overset{\mathrm{iid}}{\sim} U[0,1]$. A straightforward calculation shows that
\[
s_k^2\approx  \int_{q^{\frac{k}{a}}}^{q^{\frac{k+1}{a}}}x^{-2b}dx \approx q^{-\frac{(2b-1)k}{a}},
\]
and hence
\[
\sigma=\tau=\frac{2b-1}{2a}.
\]
Note that $b>1$ and $\tau\leq 1$ if and only if $1<b\leq a+\frac{1}{2}$. Therefore, we have
\begin{corollary}\label{coro1.4}
If $1<b\leq a+\frac{1}{2}$, for any Borel set $A\subset\mathbb{R}$, then
\begin{enumerate}[(i)]
    \item almost surely,
\begin{align*}
\hd R_{a,b,\Theta}(A)&= \min\Bigl\{2,\; \frac{2a}{2b-1}\hd A\Bigr\},\\
\hd G_{R_{a,b,\Theta}}(A)&= \min\Bigl\{\frac{2a}{2b-1}\hd A,\;\hd A+2-\frac{2b-1}{a}\Bigr\}.
\end{align*}
\item If $\hd A>\frac{2b-1}{a}$,  $R_{a,b,\Theta}(A)$ has positive  two-dimensional Lebesgue measure a.s.
\item If $\hd A>\frac{4b-2}{a}$, $R_{a,b,\Theta}(A)$ has interior points a.s.
\item If $\bd A=\hd A$, then the above two dimension formulas also hold for the box-counting dimension.
\end{enumerate}
\end{corollary}

Our results show that the Hausdorff dimension of the random  curve for $R_{2,2}$ is equal to  $4/3$ almost surely, so is the related Riemann function $\phi$ in (\ref{eq1.3}) (see Section \ref{sec5.1}). Therefore, the upper bound in (\ref{eq1.4}) should be the conjectured value for the image of $\phi$.

The paper is organized as follows. In Section \ref{sec2}, we  establish an optimal H\"{o}lder exponent for $S(x)$ to obtain the upper bound for the dimensions.   Section \ref{sec3} is the main part of the paper. We obtain a lower bound for the Sobolev dimension of $S(A)$, from which the remaining conclusions of Theorem \ref{thm1.1} follow.     Section \ref{sec4} is devoted to the Hausdorff dimension of the graphs of $S(x)$ over Borel sets. In Section \ref{sec5}, we discuss the Hausdorff dimensions of the images and graphs of other random series, such as the random Weierstrass-Mandelbrot-type series and some random one-dimensional trigonometric series. Section \ref{sec6} concludes the paper with some open problems and conjectures.

\medskip

Throughout this paper, the constants $c, c_1, c_2,  \ldots,$ etc. denote the absolute constants. We use the notation $f(x) \lesssim g(x) $ to denote that there exists a constant $C>0$ independent of $x$ such that $f(x)\leq C g(x)$ for all $x$. We write $f(x)\approx g(x)$ to mean $f(x)\lesssim g(x) \lesssim f(x)$. In some occasions, we will employ the notation $f(x)\lesssim_{\epsilon} g(x)$ to emphasize that the previous implicit constants $C$ may depend on some relevant parameter $\epsilon$.

\section{H\"{o}lder continuity and Upper bound}\label{sec2}

The goal of this section is to establish the upper bound of the Hausdorff dimensions by computing the almost sure H\"{o}lder exponent of the random series \eqref{eq1.8}. This has been a classical work in stochastic processes via the  Kolmogorov-Chentsov  theorem (see e.g. \cite[p.450]{Kle}).

\begin{theorem}[Kolmogorov-Chentsov  theorem]
 Let $ X_t (t\in (0, \infty))$ be a real-valued process. Assume for every $T > 0$, there are numbers $\alpha, \beta, C>0$ such that
 \begin{equation}\label{eq2.1}
 \mathbb{E}[|X_t-X_s|^{\alpha}]\leq C|t-s|^{1+\beta},\quad \forall s, t\in [0, T].
 \end{equation}
 Then  there is a modification $Y_t$   of $X_t$ whose paths are locally H\"{o}lder continuous of every order $\gamma\in (0, \beta/\alpha)$.
\end{theorem}
Here, the modification $Y_t$ of $X_t$ means that for every $t$, $\mathbb{P}(Y_t=X_t)=1$. Note that this event with probability 1 depends on $t$. To derive the moment estimates (\ref{eq2.1}), one needs the Marcinkiewicz-Zygmund inequality \cite[p.578]{Kaw}.

\begin{theorem}[Marcinkiewicz-Zygmund inequality]
Let $p\geq 1$ and $X_n$ be a sequence of independent random variables such that  $\mathbb{E}X_n=0$ and $\mathbb{E}|X_n|^p<\infty$. If $\sum_{n=1}^{\infty}X_n$ converges almost surely, and if the sum belongs to $L^p$, then $\sum_{n=1}^{\infty} X_n^2$ converges almost surely, belongs to $L^{p/2}$, and
\[
A_p\mathbb{E}\Bigl[\Bigl(\sum_{n=1}^{\infty}|X_n|^2\Bigr)^{\frac{p}{2}}\Bigr]\leq \mathbb{E}\Bigl|\sum_{n=1}^{\infty}X_n\Bigr|^p\leq B_p\mathbb{E}\Bigl[\Bigl(\sum_{n=1}^{\infty}|X_n|^2\Bigr)^{\frac{p}{2}}\Bigr].
\]
\end{theorem}
The inequality also holds for the complex-valued random variables by considering its real and imaginary parts via an elementary inequality $(|a|+|b|)^p\leq C_p (|a|^p+|b|^p)$  for all  $p>0$.  Our main theorem in this section is as below:
\begin{theorem}\label{thm2.3}
If $ 0<\sigma<\infty$, then almost surely, the random  series $S(x)$ in \eqref{eq1.8} is  H\"{o}lder continuous with exponent $\alpha$ for all $\alpha\in (0, \min\{\sigma, 1\})$. Consequently, almost surely, for any Borel set $A\subset\mathbb{R}$,
\begin{equation}\label{eq2.2}
\dim S(A)\leq \min\Bigl\{2,\; \frac{\dim A}{\min\{\sigma, 1\}}\Bigr\},
\end{equation}
where $\dim$ denotes the Hausdorff dimension, the lower or upper box-counting dimensions.
\end{theorem}
\begin{proof}
Let $\delta$ and $q$ be the constants in \eqref{eq1.5} and \eqref{eq1.6}. We are going to estimate $\mathbb{E}|S(x+h)-S(x)|^p$. It suffices to consider  $0<h\leq \delta/q$. Since $\{a_n\}\in l^1$  ensures that $S(x)$ is bounded, the moment estimate in the Kolmogorov-Chentsov theorem for larger $h>\delta/q$ holds automatically by considering a large constant.   Let  now $N\geq 1$ be the unique integer such that $\delta/q^{N+1}<h\leq  \delta/q^{N}$. By the  Marcinkiewicz-Zygmund inequality, for $p\geq 1$,
\[
\mathbb{E}|S(x+h)-S(x)|^p\leq B_p\mathbb{E}\Bigl(\sum_{n=1}^{\infty}|a_n|^2|X_n|^2\Delta^2_n(x,h)\Bigr)^{\frac{p}{2}}=B_p\Bigl(\sum_{n=1}^{\infty}|a_n|^2\Delta^2_n(x,h)\Bigr)^{\frac{p}{2}},
\]
where $\Delta_n(x,h)=|\phi_n(\lambda_n(x+h))-\phi_n(\lambda_n  x)|$.
The definition of $\sigma$ implies that  for any sufficiently small $\epsilon>0$,
\[
s_n\lesssim_{\epsilon} q^{-n(\sigma-\epsilon)}.
\]
Then, we see that
\begin{align*}
\sum_{n=1}^{\infty}|a_n|^2\Delta_n^2(x,h)&=\sum_{k=0}^{N-1}\sum_{q^k\leq \lambda_n<q^{k+1}}|a_n|^2\Delta_n^2(x,h)+\sum_{k=N}^{\infty}\sum_{q^k\leq\lambda_n<q^{k+1}}|a_n|^2\Delta_n^2(x,h)\\
&\leq L^2h^2\sum_{k=0}^{N-1}\sum_{q^k\leq\lambda_n<q^{k+1}}|a_n|^2 \lambda_n^2+2\sup_{n}\lVert\phi_n\rVert\sum_{k=N}^{\infty}\sum_{q^k\leq\lambda_n<q^{k+1}}|a_n|^2\\
&\leq L^2h^2\sum_{k=0}^{N-1}q^{2(k+1)}s_k^2 +2\sup_{n}\lVert\phi_n\rVert\sum_{k=N}^{\infty}s_k^2 \\
&\lesssim_{\epsilon}h^2 \sum_{k=0}^{N-1}q^{2k(1-\sigma+\epsilon)}+\sum_{k=N}^{\infty}q^{-2k(\sigma-\epsilon)}.
\end{align*}
If $0<\sigma\leq 1$, then, using the fact that $q^{-N}\approx h$,
\[
\sum_{n=1}^{\infty}|a_n|^2\Delta_n^2(x,h)\lesssim_{\epsilon} h^2 q^{2N(1-\sigma+\epsilon)} +q^{-2N(\sigma-\epsilon)}\lesssim_{\epsilon}h^{2(\sigma-\epsilon)}.
\]
On the other hand, if $ \sigma>1$, we have
\[
\sum_{n=1}^{\infty}|a_n|^2\Delta_n^2(x,h)\lesssim_{\epsilon} h^{2}+h^{2(\sigma-\epsilon)}\lesssim_{\epsilon} h^2.
\]
Thus,
\begin{equation}\label{eq2.3}
\mathbb{E}|S(x+h)-S(x)|^p\lesssim_{\epsilon,p}h^{p(\min\{\sigma, 1\}-\epsilon)}:=h^{1+\alpha_0}.
\end{equation}
The Kolmogorov-Chentsov theorem  implies that there exists a  modification $\widetilde{S}(x)$ of $S(x)$ that is  H\"{o}lder continuous with any order $0<\alpha<\alpha_0/p=\min\{\sigma, 1\}-\epsilon-1/p$.

Since $\{a_n\}\in l^1$, $S(x)$ is continuous for every sample path $\Theta$. $\mathbb{P}(S(t_n)=\widetilde{S}(t_n),\, \forall t_n\in \mathbb{Q})=1$. By the continuity of $S(t)$ and $\widetilde{S}(t)$, $\mathbb{P}(S(t)=\widetilde{S}(t),\,\forall t\in \mathbb{R})=1$.  The arbitrariness of $\epsilon$ and $p$ completes the proof of H\"{o}lder continuity.
\end{proof}

Our result established the almost sure H\"{o}lder exponents.  We remark that it is not always the case that the H\"{o}lder exponent is constant at all points. It is true that the Weierstrass function $W_{\beta,\lambda}$ has a uniform H\"{o}lder exponent $\beta$  at all points if $\beta\in(0, 1)$.  The Riemann function $R_{a,b}$ has a uniform H\"{o}lder exponent $(b-1)/a$ at all points if $1<b<a+1$ \cite{Johnsen}.  This  uniform exponent  of  Riemann function is too small to give an optimal upper bound for the dimension.  Indeed, Riemann functions possess a rich multifractal structure, and their pointwise exponents varies from points to points.  For further multifractal analysis of Riemann functions, we refer to \cite{CU2014,Jaf}.

\section{Lower bound of  Hausdorff dimension}\label{sec3}

\subsection{Sobolev dimension}

To prove the lower bound on the Hausdorff dimension in Theorem \ref{thm1.1}, let us review some basic properties about the Sobolev dimension.  The Sobolev dimension \cite[Section 5.2]{Mattila} of a finite Borel measure $\mu$ on ${\mathbb R}^d$  is defined to be
\begin{equation}\label{eq3.1}
\sd{\mu} = \sup \Bigl\{s>0: \int_{\mathbb{R}^d} |\widehat{\mu}(\xi)|^{2} |\xi|^{s-d} d\xi<\infty\Bigr\}
\end{equation}
and the Fourier transform of a measure is defined to be
\[
\widehat{\mu}(\xi) = \int_{\mathbb{R}^d} e^{-2\pi i \langle\xi, u\rangle}d\mu(u).
\]
The Sobolev dimension of a Borel set $E$ is defined to be the supremum of $s$ such that $E$ supported a finite Borel measure $\mu$ with $\sd{\mu} = s$. If $0< s< d$,  the $s$-energy of $\mu$ is
\[
I_s(\mu)=\int_{\mathbb{R}^d}\int_{\mathbb{R}^d}\frac{d\mu(u)d\mu(v)}{|u-v|^s}=c(s,d)\int_{\mathbb{R}^d}|\xi|^{s-2}|\widehat{\mu}(\xi)|^2d\xi,
\]
where $c(s,d)$ is a positive constant. As Hausdorff dimension of a Borel set $E$ can be expressed as the supremum of $s<d$ such that $I_{s}(\mu)<\infty$ for some finite Borel measure $\mu$ supported on $E$, we know that
\begin{equation}\label{eq3.2}
\hd{E} = \min \{d,\, \sd{E}\}.
\end{equation}
Unlike the energy integral, it is possible that the Sobolev dimension to go over the ambient dimension $d$. Indeed, if the set $E$ contains an interior point, then $E$ will support a non-negative infinitely many differentiable function $f$ and its Fourier transform will decay as fast as we can. Taking $\mu = f(x)dx$, $\sd{E} = \infty$. The following proposition can be found in Theorem 5.4 of \cite{Mattila}.
\begin{proposition}\label{prop3.1}
Let $E$ be a Borel set on $\mathbb{R}^d$. (i) If $\sd{E}>d$, then $E$ has positive $d$-dimensional Lebesgue measure. (ii) If $\sd{E}>2d$, then $E$ contains an interior point.
\end{proposition}

\subsection{Estimates of a product of Bessel functions}
The Bessel function of the first kind $J_0(x)$ is defined as the integral
\[
J_0(x)=\int_{0}^{1}e^{ix\sin2\pi\theta}d\theta=\int_{0}^{1}\cos(x\sin2\pi\theta)d\theta,\quad x\in \mathbb{R}.
\]
The asymptotic expansion \cite[p.34]{Mattila} of the Bessel function $J_0$ at infinity is as follows:
\[
J_0(x)=\sqrt{\frac{2}{\pi x}} \cos(x-\frac{\pi}{4})+O(x^{-\frac{3}{2}}),\quad x\to\infty.
\]
It follows \cite[p.78]{AP} that there is a global estimate for $J_0$:
\begin{equation}\label{eq3.3}
|J_0(x)|\leq (1+|x|^2)^{-\frac{1}{4}},\quad \forall \, x\in\mathbb{R}.
\end{equation}
Finally, we use the Bessel function $J_0$ to represent the second moment of the Fourier transform of the push-forward  measure induced by the random  series $S(x)$ in  \eqref{eq1.8}.

\begin{lemma}\label{lemma3.2}
Let $\nu$ be a Borel measure supported on a Borel set $A\subset\mathbb{R}$  and
\[
\mu (E) = \nu (\{x\in A\colon S(x)\in E\}), \quad \forall E\subset\mathbb{R}^2 \; \text{Borel}.
\]
Then
\[
\mathbb{E}[|\widehat{\mu}(\xi)|^2]=\iint_{A\times A}\prod_{n=1}^\infty J_0(2\pi|\xi|\varphi_n(x,y))d\nu(x)d\nu(y),
\]
where $ \varphi_n(x,y)=|a_n|\cdot |\phi_n(\lambda_nx)-\phi_n(\lambda_ny)|$.
\end{lemma}

\begin{proof}
It is more convenient to write the Euclidean inner product as $\langle x, y\rangle = \mbox{Re} (x\overline{y})$. With this notation, we can expand
\[
|\widehat{\mu}(\xi)|^2=\iint_{\mathbb{R}^2\times\mathbb{R}^2} e^{-2\pi i\langle\xi, u-v\rangle}d\mu(u)d\mu(v)=\iint_{A\times A} \prod_{n=1}^{\infty} e^{-2\pi i|\xi|\varphi_n(x,y)\cos(\gamma_n-2\pi \theta_n)}d\nu(x)d\nu(y),
\]
where $\gamma_n=\gamma_n(\xi,x,y)$ is the argument of the complex number $\xi\cdot\overline{a_n(\phi_n(\lambda_nx)-\phi_n(\lambda_ny))}$. Since  $\theta_n$ are independently uniform distributions on $[0, 1]$, we have
\begin{align*}
\mathbb{E}[|\widehat{\mu}(\xi)|^2]&=\iint_{A\times A}\prod_{n=1}^\infty\mathbb{E}\Bigl[e^{-2\pi i|\xi|\varphi_n(x,y)\cos(\gamma_n-2\pi \theta_n)}\Bigr]d\nu(x)d\nu(y)\\
&=\iint_{A\times A} \prod_{n=1}^\infty\Bigl(\int_0^1e^{-2\pi i|\xi|\varphi_n(x-y)\cos(2\pi\theta)}d\theta \Bigr)d\nu(x)d\nu(y)\\
&=\iint_{A\times A}\prod_{n=1}^\infty J_0(2\pi|\xi|\varphi_n(x,y))d\nu(x)d\nu(y),
\end{align*}
where the second equality uses the fact that the integrand is periodic with period one. The last formula is exactly our desired formula.
\end{proof}

The following lemma gives the crucial estimate for the infinite product of the Bessel functions.
\begin{lemma}\label{lemma3.3}
Let $L,q,\delta$ be the constants for the locally uniformly bi-Lipschitz functions defined in \eqref{eq1.5}-\eqref{eq1.6}. Let also $\tau$ be the constant defined in (\ref{eq1.7}).

Fix $\ell\in\N$ and $k\geq \ell$. Then for each $j$ such that  $k-\ell\le j<k-1$, $\lambda_n\in [q^j,q^{j+1})$ and $|x-y|\in [\delta q^{-(k+1)}, \delta q^{-k})$, we have
\begin{equation}\label{eq3.4}
\varphi_n(x,y)\geq  \frac{L^{-1}\delta}{q^{\ell+1}}\cdot|a_n|.
\end{equation}
Moreover, for all $\epsilon>0$ and $r>0$,
\begin{equation}\label{eq3.5}
\prod_{j=k-\ell}^{k-1}\prod_{q^j\leq \lambda_n<q^{j+1}}|J_0(2\pi r\cdot\varphi_n(x,y))|\lesssim_{\epsilon} q^{\frac{k\ell(\tau+\epsilon)}{2}}\cdot r^{-\frac{\ell}{2}}.
\end{equation}
\end{lemma}
\begin{proof}
By the locally uniformly bi-Lipschitz lower bound in \eqref{eq1.5} and the range of $\lambda_n$ and $j$,
$$
|\phi_n(\lambda_nx)-\phi_n(\lambda_ny)|\geq L^{-1} |\lambda_n||x-y|\ge \frac{L^{-1}}{q^{\ell+1}}\cdot q^{k+1}|x-y|.
$$
Using the range of $|x-y|$, we obtain
\begin{equation}\label{eq3.6}
|\phi_n(\lambda_nx)-\phi_n(\lambda_ny)|\geq\frac{L^{-1} \delta}{q^{\ell+1}}, \quad \mbox{if}\quad \frac{\delta}{q^{k+1}}<|x-y|\leq \frac{\delta}{q^k}.
\end{equation}
Hence, \eqref{eq3.4} follows.

To prove the second inequality, the bound  \eqref{eq3.3}  for the Bessel's function, we have
\begin{equation}\label{eq3.7}
|J_0(2\pi r\cdot\varphi_n(x,y))|\le (1+4\pi^2 c_1^2 r^2 |a_n|^2)^{-\frac{1}{4}}
\end{equation}
where $c_1 = L^{-1}\delta q^{-(\ell+1)}$ as in  \eqref{eq3.4}.
Note that $0\leq \tau<\infty$. By the definition of $\tau$, for all $\epsilon>0$, we have
\begin{equation}\label{eq3.8}
|s_k|\gtrsim_{\epsilon} q^{-k(\tau+\epsilon)}, \quad \forall k\ge 0.
\end{equation}
Using the elementary inequality $\prod_{k=1}^n(1+x_k)\ge 1+\sum_{k=1}^nx_k$ when all $x_k>0$,
$$
\prod_{j=k-\ell}^{k-1}\prod_{q^j\leq \lambda_n<q^{j+1}}|J_0(2\pi r\cdot\varphi_n(x,y))|\le \prod_{j=k-\ell}^{k-1}\bigl(\sum_{q^j\leq \lambda_n<q^{j+1}} 4\pi^2c_1^2r^2 |a_n|^2\bigr)^{-\frac{1}{4}} = c_2 \cdot \prod_{j=k-\ell}^{k-1} (r^{-\frac{1}{2}}s_j^{-\frac{1}{2}}).
$$
where $c_2 = (4\pi^2c_1^2)^{-\ell/4}$. Applying  \eqref{eq3.8}, the last product satisfies
$$
\prod_{j=k-\ell}^{k-1} (r^{-\frac{1}{2}}s_j^{-\frac{1}{2}})\lesssim_{\epsilon }q^{\frac{k\ell(\tau+\epsilon)}{2}}\cdot r^{-\ell/2},
$$
which shows that (\ref{eq3.5}) holds as desired.
\end{proof}

\subsection{Lower bound of Sobolev dimension} We now can establish the lower bound of the Sobolev dimension.

\begin{theorem}\label{thm3.4}
If $0\leq\tau<\infty$, then for any Borel set $A\subset\mathbb{R}$ with $\hd A>0$,
\[
\sd{S(A)} \geq \frac{\hd{A}}{\tau},\quad a.s.
\]
\end{theorem}

\begin{proof}
Let $\epsilon>0$ and let $\ell>2/(\tau+\epsilon)$ be a fixed integer, and let $\delta, q$ be the constants given by \eqref{eq1.5}-\eqref{eq1.6}. For $\alpha\nearrow\hd A$, we may assume that the diameter of $A$ does not exceed $\delta/q^{\ell}$, otherwise, we replace $A$ by a subset of $A$ with  Hausdorff dimension $>\alpha$. As the Hausdorff dimension is countable stable under union, we can always do such replacement. The Frostman's lemma \cite[Theorem 2.7]{Mattila}  tells us that there exists a  Borel measure $\nu$ supported on $A$ such that
\begin{equation}\label{eq3.9}
\nu (B(x,r))\lesssim  r^\alpha,\quad \forall x\in\mathbb{R},\; r>0.
\end{equation}
The assumption $\{a_n\}\in l^1$ ensures the continuity of $S(x)$; consequently,
$\nu$ induces a well-defined push-forward measure  $\mu=S_{\ast}\nu$ supported on $S(A)$:
\begin{equation}\label{eq3.10}
\mu(E)=\nu(\{x\in A\colon S(x)\in E\}),\quad \forall\; \text{Borel}\; E\subset \mathbb{R}^2.
\end{equation}
To obtain a lower bound for the Sobolev dimension, we need to take the expectation of the integral in (\ref{eq3.1}) and prove that if $s<\alpha/(\tau+\epsilon)$, then
\[
\int_{\mathbb{R}^2}|\xi|^{s-2}\mathbb{E}[|\widehat{\mu}(\xi)|^2]d\xi<\infty.
\]

Let
\[
A_k=\Bigl\{(x,y)\in A\times A: \frac{\delta}{q^{k+1}}< |x-y|\leq\frac{\delta}{q^k}\Bigr\}
\]
for $k\geq \ell$ (since we assume the diameter of $A\leq \delta/q^{\ell}$). Let $A_{k,x}=\{y\in A\colon (x, y)\in A_k\}$.  \eqref{eq3.9} implies that $\nu(A_{k,x})\lesssim q^{-\alpha k}$ for all $x\in A$. As $\nu$ is a $\sigma$-finite measure, by the Fubini's theorem, we have
\[
(\nu\times\nu)(A_k)=\int \nu(A_{k, x}) d\nu(x) \lesssim q^{-\alpha k}
\]
and $\nu\times \nu(\{(x, y)\in A\times A\colon x=y\})=0$.  Then we apply Lemma \ref{lemma3.2} to measure $\mu$ and decompose the energy into two parts via the polar coordinates:
\begin{equation}\label{eq3.11}
\int_{\mathbb{R}^2}|\xi|^{s-2}\mathbb{E}[|\widehat{\mu}(\xi)|^2]d\xi=2\pi\sum_{k=\ell}^{\infty}\iint_{A_k}(B_{k}(x,y)+U_{k}(x,y))d\nu(x)d\nu(y),
\end{equation}
where
\[
B_{k} (x,y)  = \int_{0}^{q^{(\tau+\epsilon)k}} r^{s-1}\prod_{n=1}^\infty J_0(2\pi r\cdot\varphi_n(x,y))dr,\quad  U_{k} (x,y)  = \int_{q^{(\tau+\epsilon)k}}^{\infty} r^{s-1}\prod_{n=1}^\infty J_0(2\pi r\cdot\varphi_n(x,y))dr.
\]
For the first integral, as $|J_0|\leq 1$,
\begin{equation}\label{eq3.12}
|B_{k}(x,y)|\leq \int_{0}^{q^{(\tau+\epsilon)k}}r^{s-1}dr=\frac{q^{sk(\tau+\epsilon)}}{s}.
\end{equation}
Now, we  estimate the second integral  by rearranging the infinite product into blocks $\{n: q^j\leq \lambda_n< q^{j+1}\}$ and keeping the product of $\ell$ blocks. Using \eqref{eq3.5}  in Lemma \ref{lemma3.3},
\begin{align*}
|U_k(x,y)|\leq \int_{q^{k(\tau+\epsilon)}}^{\infty}r^{s-1}\prod_{j=k-\ell}^{k-1}\prod_{q^j\leq \lambda_n<q^{j+1}}|J_0(2\pi r\cdot \varphi_n(x,y))|~dr \lesssim_{\epsilon}q^{\frac{k\ell(\tau+\epsilon)}{2}}\int_{q^{k(\tau+\epsilon)}}^{\infty} r^{s-1-\ell/2}~ dr.
\end{align*}
We  notice that
\begin{equation}\label{eq3.13}
s<\frac{\alpha}{\tau+\epsilon}< \frac{1}{\tau+\epsilon}<\frac{\ell}{2}
\end{equation}
by our choice of $\ell$.  (\ref{eq3.13}) ensures the finiteness of the integral. Integrating it results in
\begin{equation}\label{eq3.14}
|U_k(x,y)|\lesssim_{\epsilon} q^{sk(\tau+\epsilon)}.
\end{equation}
Putting estimates (\ref{eq3.12}) and (\ref{eq3.14}) together into (\ref{eq3.11}), we have  if $s<\alpha/(\tau+\epsilon)$, then
\[
\int_{\mathbb{R}^2}|\xi|^{s-2}\mathbb{E}[|\widehat{\mu}(\xi)|^2]d\xi\lesssim_{\epsilon} \sum_{k=\ell}^{\infty}q^{sk(\tau+\epsilon)}\cdot\nu\times\nu(A_k)\lesssim_{\epsilon} \sum_{k=\ell}^{\infty} q^{(s(\tau+\epsilon)-\alpha)k} <+\infty.
\]
This shows that
\[
\sd S(A)\geq \frac{\hd A}{\tau+\epsilon}.
\]
The proof is now complete after letting $\epsilon\to 0$. \qedhere
\end{proof}

With this lower bound on the Sobolev dimension, we  complete  the proof for Theorem \ref{thm1.1}.

\begin{proof}[Proof of Theorem \ref{thm1.1}]
(i) If $0<\sigma<\tau<\infty$, there is nothing to prove if $\hd{A}=0$, so we may assume $\hd{A}>0$.  It follows from  \eqref{eq2.2}, \eqref{eq3.2} and Theorem \ref{thm3.4}   that
\[
\min\Bigl\{2,\, \frac{\hd A}{\tau}\Bigr\}\leq \hd S(A)\leq \min\Bigl\{2,\, \frac{\hd A}{\min\{\sigma, 1\}}\Bigr\},\quad a.s.
\]
(ii) If $\hd A>2\tau>0$, then $\sd  S(A)\geq 2$.  By Proposition \ref{prop3.1} (i), $S(A)$ has two-dimensional Lebesgue measure a.s.  (iii) Similarly, when $\hd A>4\tau>0$, we know that $\sd  S(A) > 4 $. By Proposition \ref{prop3.1} (ii), $S(A)$ has interior points a.s. (iv) If $\tau=0$, we  choose $\epsilon<\frac{1}{4}\hd A$, then by  Theorem \ref{thm3.4}, $\sd  S(A) > 4 $. Hence $S(A)$ has interior points a.s.
\end{proof}
\begin{remark}
The at most exponential growth condition \eqref{eq1.6} of $\{\lambda_n\}$  ensures that all blocks $\{n\geq 1: q^k\leq \lambda_n<q^{k+1}\}$ are non-empty, thereby obtaining an upper bound for $U_k$. If $\{\lambda_n\}$ has super-exponential growth, then infinitely many blocks will become empty, and the product terms in \eqref{eq3.5} will become trivially constant. This may cause  the integral in $U_k$ to be infinite for all $s > 0$.
\end{remark}

\subsection{Further Examples} We end this section by demonstrating some similar series of the same nature. They are also studied in literature.

\begin{example}
Let $\lVert t \rVert=\min\{ |t-n|\colon n\in\mathbb{Z}\}$ be the distance from $t$ to the nearest integer. If we replace the cosine function in the real-valued Weierstrass function with the distance function $\lVert t \rVert$, then a Takagi-type function is obtained \cite{AK}.  It is easy to verify that  $\phi(t)=\lVert t\rVert+i\sin(2\pi t)$ is locally uniformly  bi-Lipschitz. See Figure \ref{fig5} for the curves $W_{\beta, \lambda}^{\phi}([0, 1])$.
\end{example}

\begin{figure}[ht]
\begin{subfigure}{0.3\textwidth}
\centering\includegraphics[height=3cm]{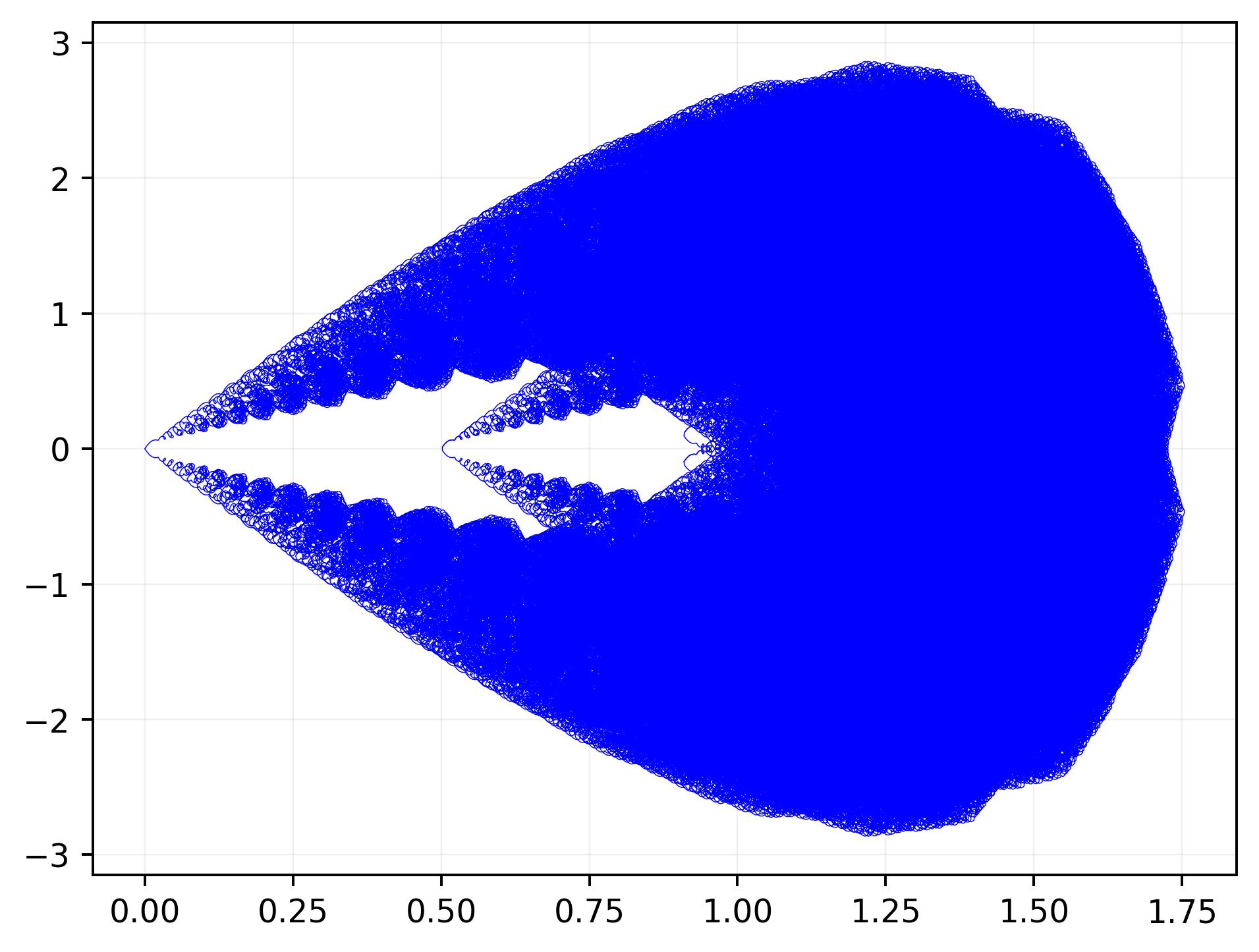}
\caption*{$\lambda=2, \beta=0.3$}
\end{subfigure}
\hfill%
\begin{subfigure}{0.3\textwidth}
\centering\includegraphics[height=3cm]{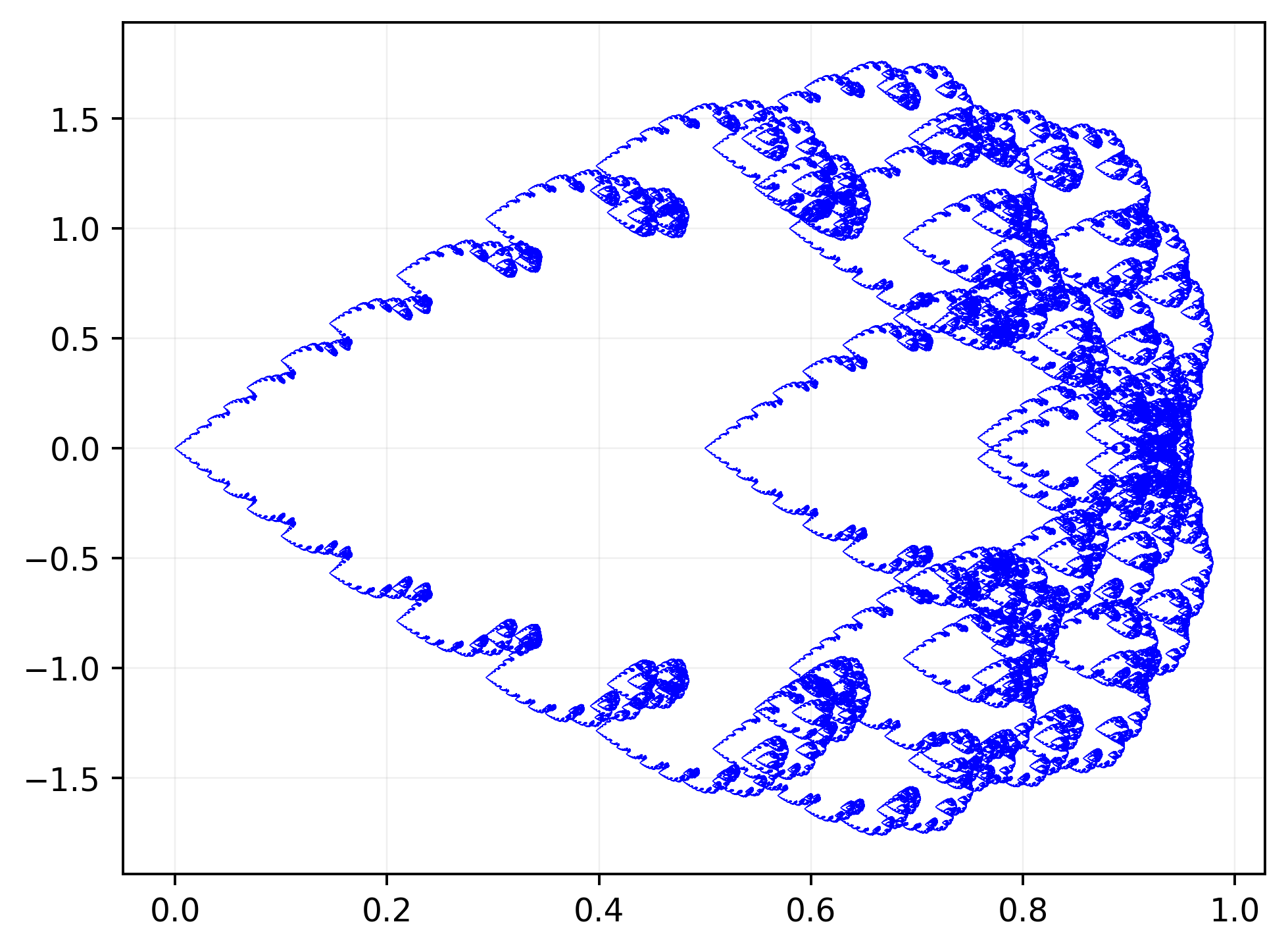}
\caption*{$\lambda=2, \beta=0.6$}
\end{subfigure}
\hfill%
\begin{subfigure}{0.3\textwidth}
\centering\includegraphics[height=3cm]{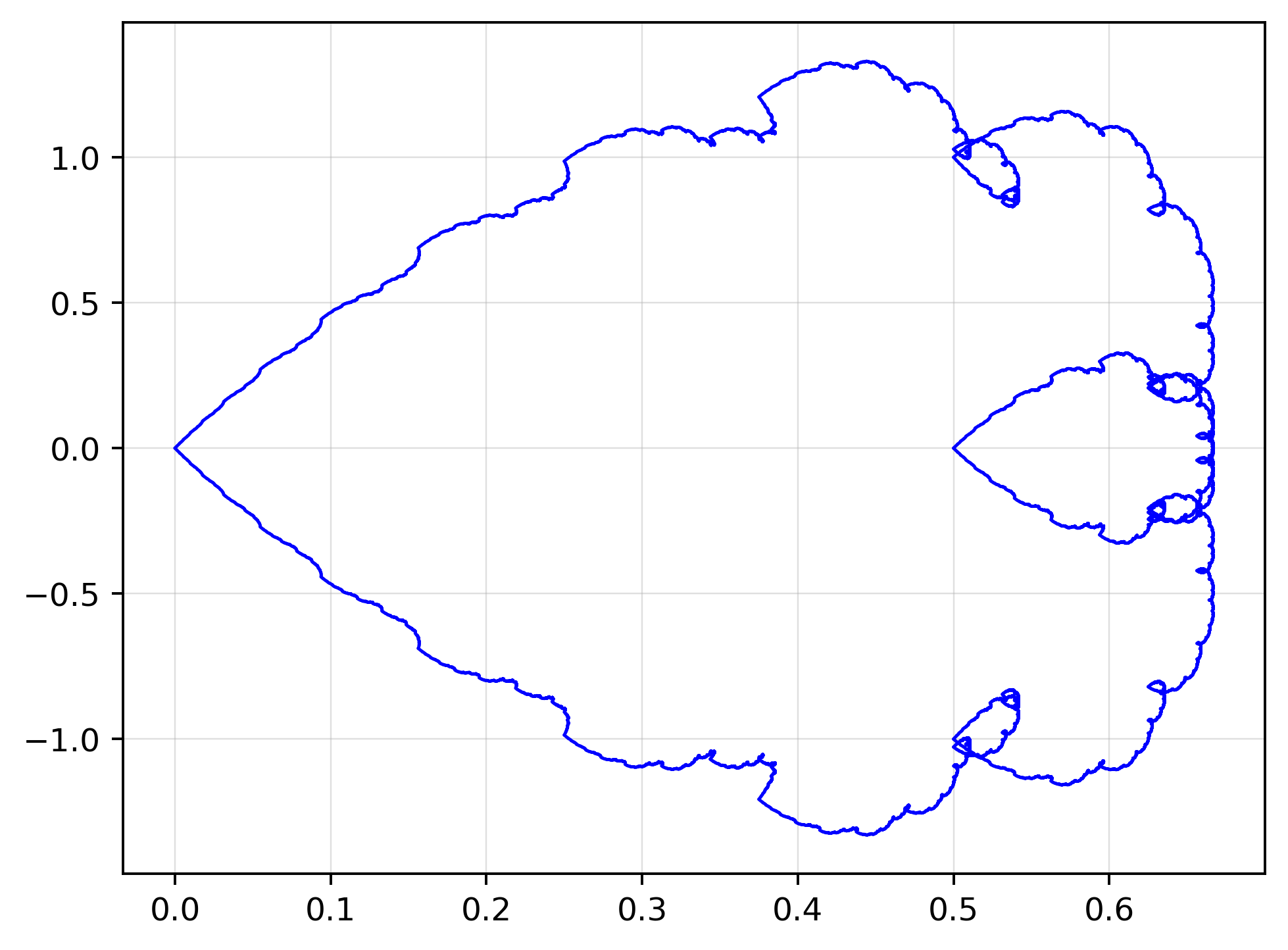}
\caption*{$\lambda=2, \beta=1$}
\end{subfigure}
\caption{Curves $W_{\beta,\lambda}^{\phi}([0, 1])$ for $\phi(t)=\lVert t\rVert+i\sin(2\pi t)$.} \label{fig5}
\end{figure}

\begin{example}
(i) Hardy \cite{Hardy} asserted that the series $\sum_{n=1}^{\infty}n^{-2}\cos b^nx$ is not H\"{o}lder continuous at any point. This complex series falls into the case  $\tau=0$ in Theorem \ref{thm1.1} (see Figure \ref{fig6}). It is known that the  Hadamard gap series $\sum_{k=1}^{\infty}k^{-p} e^{i n_k x} \,(p>1, \inf_k n_{k+1}/n_k>1)$ provides an example of a space-filling curve in   \cite{KWW}.

(ii) For any $\alpha>0$, there exists a Riemann-type series $\sum_{n=1}^{\infty}a_n X_n e^{2\pi i n^{\alpha}x}$ such that $\{a_n\}\in l^1$ and $\sigma=\tau=0$. The construction is as follows: for each block $\{n\geq 1: q^k\leq n^{\alpha}< q^{k+1}\}$, choose the coefficients $a_n$ to be a geometric sequence with first term $k^{-2}$ and common ratio $1/2$.  Then  $s_k \approx k^{-2}$, $\{a_n\}\in l^1$, and $\sigma=\tau=0$.
\end{example}

\begin{figure}[ht]
\centering\includegraphics[width=3.5cm]{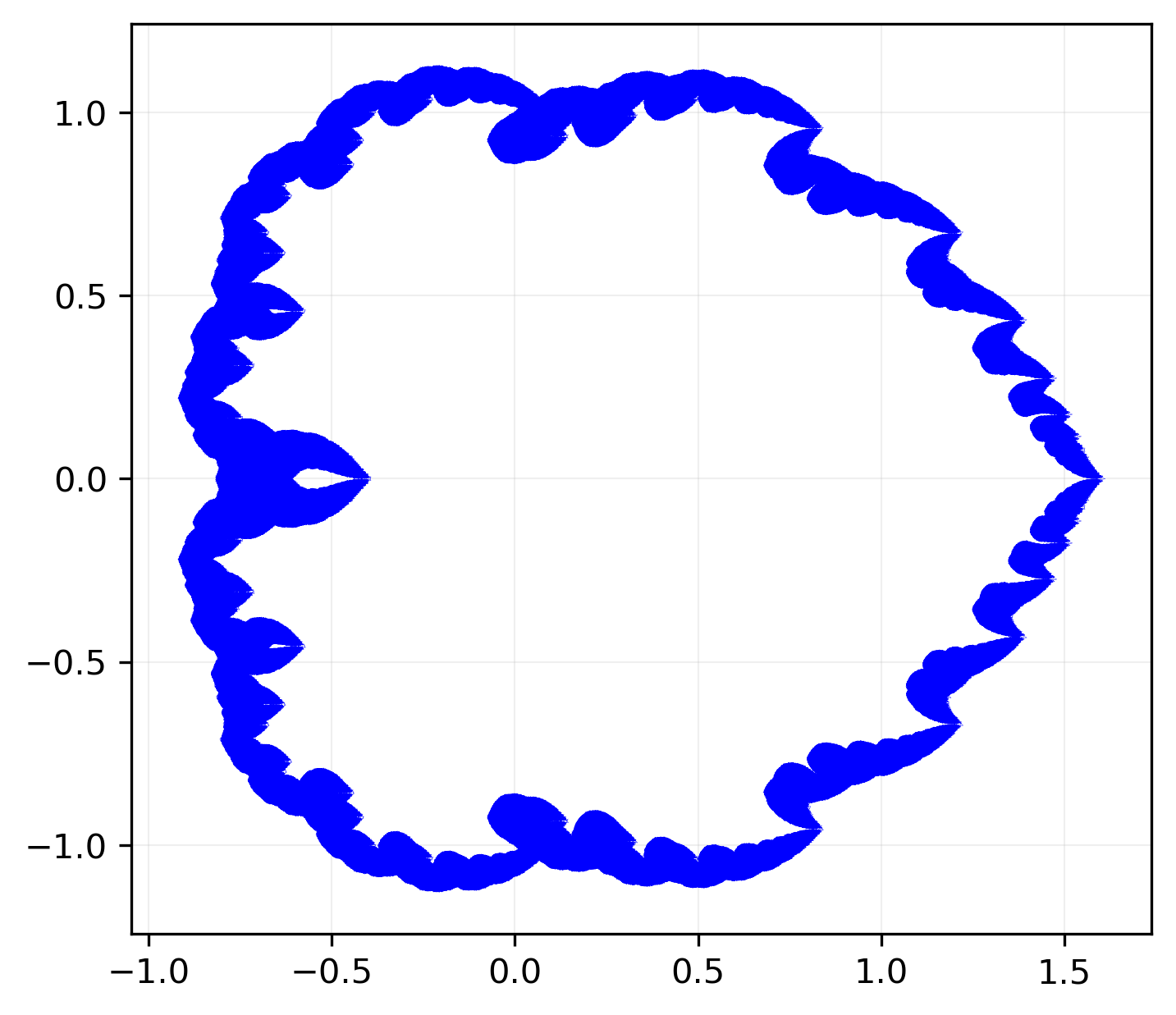}
\caption{Image of $[0, 1]$ under  the series $\sum_{n=1}^{\infty}n^{-2}e^{2\pi i2^nx}$ with $\tau=0$.}\label{fig6}
\end{figure}

\section{Graphs}\label{sec4}
In this section, we study the Hausdorff dimension of the graph of the random series $S(x)$ in \eqref{eq1.8}.  Let $X$ be a random vector in $\mathbb{R}^d$. To be consistent with the notation used earlier, we define the characteristic function of $X$ as
\[
\psi_X(\xi)=\mathbb{E}\bigl[e^{-2\pi i\langle\xi, X\rangle}\bigr],\quad \xi\in\mathbb{R}^d.
\]
By a standard argument in probability theory, namely the inversion formula for characteristic functions, if the characteristic function of a random variable is integrable, then the random variable admits a continuous bounded density function \cite[Theorem 1.3.6]{Sasvari} (see also \cite[Theorem 3.4]{Mattila}).   More precisely, we have the following theorem.
\begin{theorem}
If $\psi_X(\xi)\in L^1(\mathbb{R}^d)$, then $X$ has a bounded continuous density function $f_X$   given by  the inverse Fourier transform
\begin{equation}\label{eq4.1}
f_X(u)=\int_{\mathbb{R}^d}\psi_X(\xi)e^{2\pi i\langle\xi, u\rangle}d\xi,\quad u\in\mathbb{R}^d.
\end{equation}
\end{theorem}

 For any given $x\neq y$, consider $S(x)-S(y)$ as a complex random variable.  Its characteristic function is easily computed as
\[
\psi_{x,y}(\xi)=\mathbb{E}\bigl[e^{-2\pi i\langle\xi, S(x)-S(y)\rangle}\bigr]=\prod_{n=1}^{\infty}J_0\bigl(2\pi|\xi|\cdot|a_n(\phi_n(\lambda_nx)-\phi_n(\lambda_ny))|\bigr).
\]
Let $f_{x,y}(u)$ be the density function of $S(x)-S(y)$, its upper bound can be given by \eqref{eq4.1}.

\begin{proposition}\label{prop4.2}
Let $0<|x-y|\leq \delta/q^{5}$. If $0\leq \tau<\infty$, then for any sufficiently small $\epsilon>0$,
\[
|f_{x,y}(u)|\lesssim_{\epsilon} |x-y|^{-2(\tau+\epsilon)},\quad u\in\mathbb{R}^2.
\]
\end{proposition}
Recall that both $\delta$ and $q$ are constants given by \eqref{eq1.5} and \eqref{eq1.6}.
\begin{proof}
Take $\ell = 5$ in Lemma \ref{lemma3.3}. Let  $k\geq 5$   such that $\delta/q^{k+1}<|x-y|\leq \delta/q^{k}$. We first turn into polar coordinates. Then, using \eqref{eq3.5} in Lemma \ref{lemma3.3},  we obtain
\begin{align*}
\int_{\mathbb{R}^2} |\psi_{x,y}(\xi)|d\xi&=2\pi\Bigl(\int_{0}^{q^{k(\tau+\epsilon)}}+\int_{q^{k(\tau+\epsilon)}}^{\infty}\Bigr)|\psi_{x,y}(r)|~rdr\\
&\lesssim_{\epsilon}\int_{0}^{q^{k(\tau+\epsilon)}}r~dr+\int_{q^{k(\tau+\epsilon)}}^{\infty}  r \cdot \prod_{j=k-5}^{k-1}\prod_{q^j\leq \lambda_n<q^{j+1}}\Bigl|J_0\bigl(2\pi r\cdot|a_n(\phi_n(\lambda_nx)-\phi_n(\lambda_ny))|\bigr)\Bigr|~ dr\\
&\lesssim_{\epsilon} q^{2k(\tau+\epsilon)}+q^{\frac{5k(\tau+\epsilon)}{2}}\int_{q^{k(\tau+\epsilon)}}^{\infty} r^{1-\frac{5}{2}}dr\lesssim_{\epsilon}|x-y|^{-2(\tau+\epsilon)}<\infty.
\end{align*}
The desired upper bound follows from  $|f_{x,y}(u)|\leq \lVert\psi_{x,y}(\xi)\rVert_{L^1}$ and (\ref{eq4.1}).
\end{proof}

Let $G_f(x)=\{(x, f(x))\colon x\in\mathbb{R}\}$ be the graph function of the  function $f:\mathbb{R}\to\mathbb{R}^d$.  By Theorem \ref{thm2.3}, if $0<\sigma<\infty$, then $S(x)$ is almost surely H\"{o}lder continuous with order $\min\{1,\sigma\}-\epsilon$ for any sufficiently small $\epsilon>0$.  Let $A\subset\mathbb{R}$ be a Borel set, it follows from   \cite[Theorem 10.6]{Kahane} that the Hausdorff dimension of $G_S( A)$ is bounded above by
\begin{equation}\label{eq4.2}
\hd  G_S( A) \le \bd  G_S( A)\leq \min\Bigl\{\frac{\dim A}{\min\{1,\sigma\}},\; \dim A+2(1-\min\{1,\sigma\})\Bigr\},\quad a.s.
\end{equation}

Note that the image $S(A)$ is actually the projection of the graph  $G_S( A)$.  The projection
\begin{align*}
P: G_S( A)& \longrightarrow S(A)\\
(x,S(x))& \longmapsto S(x)
\end{align*}
is Lipschitz, then
\[
\dim  S (A)=\dim P(G_S(A))\leq \dim  G_S( A).
\]
 Theorem \ref{thm1.1} implies that if $\tau\geq \frac{1}{2}\hd A$,
\begin{equation}\label{eq4.3}
 \hd G_S( A)\geq \hd S(A)\geq \frac{\hd A}{\tau}, \quad a.s.
\end{equation}
With the estimate of density function, we can determine the lower bound for  $0< \tau<\frac{1}{2}\hd A$.

\begin{proof}[Proof of Theorem \ref{thm1.2}] (i) By \eqref{eq4.2} and \eqref{eq4.3}, we only need to consider $0\leq  \tau<\frac{1}{2}\hd A$. Similar to Theorem \ref{thm3.4}, we assume that the diameter of $A$ does not exceed $\delta/q^{5}$. We choose $\alpha\nearrow\hd A$, and choose $\epsilon\searrow 0$ such that
\[
\tau+\epsilon<\frac{\alpha}{2}.
\]
Let $\nu$ be the Borel measure supported on $A$ given by \eqref{eq3.9}. Define the push-forward measure as
\[
\mu_1(E)=\nu(\{x\in A: G_S(x)\in E\}),\quad \forall\, \text{Borel}\, E\subset\mathbb{R}^3.
\]
It follows from Proposition \ref{prop4.2} that
\begin{align*}
\mathbb{E}[I_s(\mu_1)]&=\iint \mathbb{E}\bigl[(|x-y|^2+|S(x)-S(y)|^2)^{-\frac{s}{2}}\bigr]d\nu(x)d\nu(y)\\
&=\iint \int_{\mathbb{R}^2}\frac{f_{x,y}(u)}{(|x-y|^2+|u|^2)^{\frac{s}{2}}}du d\nu(x)d\nu(y)\lesssim_{\epsilon}\iint \Bigl(\int_{\mathbb{R}^2}\frac{|x-y|^{-2(\tau+\epsilon)}}{(|x-y|^2+|u|^2)^{\frac{s}{2}}}du \Bigr)d\nu(x)d\nu(y)\\
&=2\pi\iint|x-y|^{-2(\tau+\epsilon)} \Bigl(\int_{0}^{\infty}\frac{r  dr}{(|x-y|^2+r^2)^{\frac{s}{2}}} \Bigr)d\nu(x)d\nu(y).
\end{align*}
Let $A_k=\{(x,y)\in A: \delta/q^{k+1}<|x-y|\leq \delta/q^k\}$. If $2<s<2(1-\tau-\epsilon)+\alpha$, by the change of variables $t=r^2/|x-y|^2$, we obtain
\begin{align*}
\mathbb{E}[I_s(\mu_1)]&\leq \pi\sum_{k=5}^{\infty}\iint_{A_k}|x-y|^{2-2(\tau+\epsilon)-s}d\nu(x)d\nu(y)\int_{0}^{\infty}\frac{dt}{(1+t)^{\frac{s}{2}}}\\
&\lesssim_{\epsilon}\sum_{k= 5}^{\infty}\iint_{A_k}|x-y|^{2-s-2(\tau+\epsilon)}d\nu(x)d\nu(y)\lesssim_{\epsilon} \sum_{k= 5}^{\infty}q^{k(2\tau+2\epsilon-2+s-\alpha)}<\infty,
\end{align*}
where the last integral on the right-hand side of the first inequality is finite for $s>2$.
This implies that
\[
\hd G_S( A)\geq 2(1-\tau-\epsilon)+\alpha.
\]
The arbitrariness of $\epsilon$ and $\alpha$ completes the proof.

(ii) If $\tau=0$, the above argument shows that
\[
\hd G_S( A)\geq 2+\hd A.
\]
The upper bound is easily obtained since $G_S( A)=\{(x, G(x))\colon x\in A\}\subset A\times \mathbb{R}^2$.
\end{proof}

\begin{remark}
Note that $\hd A+2-2\tau<\frac{1}{\tau}\hd A$ if and only if $\tau>1$ or $\tau<\frac{1}{2}\hd A$. To avoid the minimum value of the lower bound in Theorem \ref{thm1.2} being negative, we restrict $\tau\leq 1$. If $\tau>1$, the lower bound is given by \eqref{eq4.3}.
\end{remark}

\section{Other random series}\label{sec5}

\subsection{Random Weierstrass-Mandelbrot-type series}\label{sec5.1}
We say that a function $f$ has the \emph{scaling property} if there exist some constants $a, b>0$ such that $f(ax)=bf(x)$ for all $x$. One can readily check that the Weierstrass function $W_{\beta, \lambda}(x)=\sum_{n=1}^{\infty}\lambda^{-\beta n}e^{2\pi i\lambda^nx}$ does not have {scaling property}. Mandelbrot \cite[p.328]{Mandelbrot}  discovered  that {scaling property} could be achieved  by adding the low-frequency terms corresponding to negative indices:
\[
W(x)=\sum_{n=-\infty}^{\infty}\lambda^{-\beta n} (1-e^{2\pi i\lambda^n x}),\quad \lambda>1, \quad 0<\beta<1, \quad x\in \mathbb{R}.
\]
This function is known as the Weierstrass-Mandelbrot function and satisfies the scaling  property  $W(\lambda x)=\lambda^{\beta} W(x)$.

Berry and Lewis \cite{BL}  observed that there is some connection between the limit of the Steinhaus Weierstrass-Mandelbrot function and the fractional Brownian motion; Szulga and Molz \cite{SM} proved the Berry-Lewis's limit theorem.  Szulga \cite{Szulga} showed that  the Steinhaus  Weierstrass-Mandelbrot function and the fractional Brownian motion share the Hausdorff dimension of paths and graphs, i.e.,
\[
\hd W([0, 1])=\min\{2,\; \beta^{-1}\},\quad \hd G_W([0, 1])=\min\{\beta^{-1},\; 3-2\beta\}.
\]

In this section, we replace $e^{2\pi i x}$ by boundedness, uniformly locally bi-Lipschitz functions  $\phi_n(x):\mathbb{R}\to\mathbb{C}$ with $\phi_n(0)=0\,(n\in\mathbb{Z})$,  and still denote the random Weierstrass-Mandelbrot-type series by $W(x)$:
\begin{equation}\label{eq5.1}
W(x)=\sum_{n=-\infty}^{\infty}\lambda^{-\beta n} X_n\phi_n(\lambda^nx),\quad \lambda>1,\;\; 0<\beta<1,\;\;x\in\mathbb{R},
\end{equation}
where $\{X_n\}$ is still a Steinhaus sequence. If $\phi_n(t)=1-e^{2\pi i t}$, $W(x)$ is the Steinhaus Weierstrass-Mandelbrot function in \cite{Szulga}. We  can generalize  the dimension results in \cite{Szulga} to Corollary \ref{coro1.3} for  the random Weierstrass-Mandelbrot-type series \eqref{eq5.1}. First, we obtain the H\"{o}lder property.

\begin{proposition}\label{prop5.1}
$W(x)$ is $\beta$-H\"{o}lder continuous for every  sample path $\Theta$.
\end{proposition}
\begin{proof}
Denote
\[
W^{+}(x)=\sum_{n=0}^{+\infty}\lambda^{-\beta n}X_{n}\phi_{n}(\lambda^n x),\quad W^{-}(x)=\sum_{n=1}^{+\infty}\lambda^{\beta n} X_{-n}\phi_{-n}(\lambda^{-n} x).
\]
The condition \eqref{eq1.5} and $\phi_n(0)=0$ imply that $W^{-}(x)$ converges at every  $x$. Let $\delta$ be given by \eqref{eq1.5} and $0<|x-y|\leq \delta$, then
\[
|W^{-}(x)-W^{-}(y)|\leq \sum_{n=1}^{+\infty}\lambda^{\beta n} \cdot  \big|\phi_{-n}(\lambda^{-n}x)-\phi_{-n}(\lambda^{-n}y)\big|\lesssim |x-y|\sum_{n=1}^{\infty} \lambda^{-(1-\beta)n}.
\]
Since the last sum  converges,  $W^{-}$ is Lipschitz.

Let  $N\geq 1$ be the integer such that $\delta/\lambda^{N+1}<|x-y|\leq \delta/\lambda^N$, then
\[
|W^{+}(x)-W^{+}(y)|\leq \sum_{n=0}^N\lambda^{-\beta n}|\phi_n(\lambda^nx)-\phi_n(\lambda^ny)|+2\lVert\phi_n\rVert\sum_{n=N+1}^{\infty}\lambda^{-\beta n}\lesssim |x-y|^{-\beta}.
\]
The conclusion follows.
\end{proof}

It is well known that the graph of a function has the same Hausdorff dimension as that of its sum with any Lipschitz functions \cite{MW}.     We now prove that adding a independent Lipschitz function to $S(x)$ in \eqref{eq1.8} does not change the dimension of its images.

\begin{theorem}\label{thm5.2}
Let $f\colon\mathbb{R}\to\mathbb{C}$  be a function independent of $S(x)$ (it may be either random or deterministic).
If $f$ and $S+f$ have the same H\"{o}lder  exponent, then Theorems \ref{thm1.1} and \ref{thm1.2} still hold for  $S(x)+f(x)$.
\end{theorem}
\begin{proof}
(i)  Denote $g(x)=S(x)+f(x)$. Let $\nu$ and $\mu$  be the measures given by Theorem \ref{eq3.3}, and let $\mu_g=g_{\ast}\nu$ be the push-forward measure. Then by Theorem \ref{thm3.4}, if $s<\hd A/\tau$,
\begin{align*}
\int|\xi|^{s-2}\mathbb{E}[|\widehat{\mu_g}(\xi)|^2]d\xi&=\int|\xi|^{s-2}\iint \mathbb{E}\bigl[e^{-2\pi i\langle\xi, g(x)-g(y)}\bigr]d\xi\\
&=\int|\xi|^{s-2}\iint \mathbb{E}\bigl[e^{-2\pi i\langle\xi, f(x)-f(y)}\bigr] \mathbb{E}\bigl[e^{-2\pi i\langle\xi, S(x)-S(y)}\bigr]d\nu(x)d\nu(y)d\xi\\
&\leq \int|\xi|^{s-2}\iint \Bigl|\prod_{n=1}^{\infty}J_0(2\pi|\xi|\cdot|a_n(\phi_n(\lambda_nx)-\phi_n(\lambda_ny))|)\Bigr|d\nu(x)d\nu(y)d\xi<\infty.
\end{align*}
This implies that
\[
\hd g(A)\geq \min\Bigl\{2,\;\frac{\hd A}{\tau}\Bigr\}.
\]
The upper bound for the dimension follows from the fact that $S(x)+f(x)$ and $S(x)$ share the same H\"{o}lder exponent.

(ii) The upper bound of dimension of the graph of $S(x)+f(x)$  remains unchanged, as the H\"{o}lder exponent does not change.  Similar to (i), Proposition \ref{prop4.2} remains valid for $S(x)+f(x)$; consequently, the lower bound is also preserved.
\end{proof}
Applying this theorem, we obtain the dimensions of the random Weierstrass-Mandelbrot-type series as well as of the random generalization of the Riemann function.
\begin{corollary}
Corollary \ref{coro1.3} still holds for the  random Weierstrass-Mandelbrot-type series $W(x)$ with $0<\beta<1$ and $\lambda>1$.
\end{corollary}
\begin{proof}
By Proposition \ref{prop5.1},  $W(x)=W^{+}(x)+W^{-}(x)$ and $W^{+}(x)$  have the same H\"{o}lder exponent. This corollary follows from  Theorem \ref{thm5.2}.
\end{proof}
\begin{corollary}
Consider the random generalization of the Riemann function
\[
\phi_{\Theta}(t)=\sum_{n\in\mathbb{Z}} \frac{(e^{-4\pi^2in^2t}-1)X_n}{-4\pi^2n^2}=-\frac{1}{4\pi^2}R_{2,2,\Theta^+}(-2\pi t)-\frac{1}{4\pi^2}R_{2,2,\Theta^{-}}(-2\pi t)+itX_0+\sum_{n\neq 0}\frac{X_n}{4\pi^2n^2},
\]
where $X_n$ are Steinhaus sequence and $\Theta^{+}=\{\theta_n: n\geq 1\}, \Theta^{-}=\{\theta_n: n<0\}$. Then
\[
\hd \phi_{\Theta}(A)=\frac{4}{3}\hd A,\quad a.s.
\]
\end{corollary}
\begin{proof}
Note that linear transformations do not change the Hausdorff dimension of a set. $R_{2,2,\Theta^+}$ and $R_{2,2,\Theta^-}$ have the same probability distribution and  share the same H\"{o}lder exponent. Since $it$ is Lipschitz, this corollary follows from  Theorem \ref{thm5.2}.
\end{proof}

\subsection{One-dimensional trigonometric series}\label{sec5.2}
If we check the proof of Theorem \ref{thm1.1} more carefully, we can see that the argument does not actually require the series $S(x)$ to be two-dimensional. The core of the proof of Theorem \ref{thm1.1} consists of the estimates for the Bessel function $J_0$ in Lemma \ref{lemma3.3}  and the positive lower bound in \eqref{eq3.6}. The Bessel functions arise from Steinhaus random variables, while the lower bound in \eqref{eq3.6} comes from the lower Lipschitz estimate \eqref{eq1.5} of $\phi_n$. In fact, the two conditions can still be satisfied by certain one-dimensional  functions  $\phi_n\colon\mathbb{R}\to\mathbb{R}$.
 For example, suppose that $\phi_n$ is a family of 1-periodic, uniformly bounded and locally uniformly Lipschitz
functions  such that the oscillatory integrals of the  phases $\phi_n(x+\theta)-\phi_n(y+\theta)$  satisfy a certain  uniform polynomial decay; other conditions may refer to \cite{Hunt}. In this subsection, we will illustrate this with the random sine  series:
\begin{equation}\label{eq5.2}
S_1(x)=\sum_{n=1}^{\infty}a_n\sin2\pi(\lambda_nx+\theta_n),\quad a_n\in\mathbb{R},
\end{equation}
where $\theta_n\overset{\mathrm{iid}}{\sim} U[0,1]$. The same also holds for the cosine series, since they share the same distribution.

\begin{theorem}\label{thm5.5}
If $0<\sigma\leq \tau<\infty$, for any Borel set $A\subset\mathbb{R}$, then (i) almost surely,
\[
\min\Bigl\{1,\; \frac{\hd A}{\tau}\Bigr\}\leq \hd S_1(A)\leq \min\Bigl\{1,\; \frac{\hd A}{\min\{\sigma, 1\}}\Bigr\}.
\]
(ii) If $0<\tau<\hd A$, then  $S_1(A)$ has positive one-dimensional Lebesgue measure a.s. (iii) If $0<\tau<\frac{1}{2}\hd A$, then  $S_1(A)$ has interior a.s. (iii) If $\tau=0$ and  $\hd A>0$, then $S_1(A)$ has interior a.s.
\end{theorem}
\begin{proof}
The proof is essentially the same as  that of Theorem  \ref{thm1.1}, we will only indicate the key steps. It suffices to prove that  Theorem \ref{thm3.4} also holds for $S_1$, i.e.,
\begin{equation}\label{eq5.3}
\sd S_1(A)\geq \frac{\hd A}{\tau},\quad a.s.
\end{equation}
Let $\epsilon>0$ and $\ell>2/(\tau+\epsilon)$ be a fixed integer. Assume that $\hd A>0$ and  $\mathrm{diam}\, A\leq \delta/q^{\ell}$.   $\forall\,\alpha\nearrow \hd A$, let  $\nu$ be a Borel measure supported on $A$ satisfying \eqref{eq3.9}. The push-forward measure supported on $S_1(A)$ is given by
\[
\mu_2(E)=\nu(\{x\in A\colon S_1(x)\in E\}),\quad \forall\, \mathrm{Borel}\, E\subset\mathbb{R}.
\]
Note that $\widehat{\mu_2}(\xi)=\int_{\mathbb{R}}e^{-2\pi i\xi S_1(x)}d\nu(x)$ and
\[
\sin2\pi(\lambda_nx+\theta_n)-\sin2\pi(\lambda_ny+\theta_n)=2\sin\pi\lambda_n(x-y)\cos(2\pi\theta_n+\pi\lambda_n(x+y)).
\]
Let  $\widetilde{\varphi}_n(x,y)=2a_n\sin\pi\lambda_n(x-y)$, we have
\[
\mathbb{E}[|\widehat{\mu_2}(\xi)|^2]=\iint_{A\times A} \prod_{n=1}^{\infty} J_0(2\pi \xi \widetilde{\varphi}_n(x,y))d\nu(x)d\nu(y).
\]
Therefore, we decompose the energy into two parts:
\[
 \int_{\mathbb{R}}|\xi|^{s-1}\mathbb{E}[|\widehat{\mu_2}(\xi)|^2]d\xi:=\sum_{k=\ell}^{\infty}\iint_{A_k}(\widetilde{B}_k+\widetilde{U}_k)d\nu(x)d\nu(y),
\]
where $A_k=\{(x,y)\in A\times A: 1/(2q^{k+1})<|x-y|\leq 1/(2q^{k})\} $ and
\[
\widetilde{B}_k=\int_{|\xi|\leq q^{k(\tau+\epsilon)}} |\xi|^{s-1} \prod_{n=1}^{\infty} J_0(2\pi \xi \widetilde{\varphi}_n(x,y))d\xi,\quad \widetilde{U}_{k} = \int_{|\xi|>q^{(\tau+\epsilon)k}} |\xi|^{s-1}\prod_{n=1}^\infty J_0(2\pi r\widetilde{\varphi}_n(x,y))d\xi.
\]
Obviously,
\[
|\widetilde{B}_k|\leq \int_{|\xi|\leq q^{k(\tau+\epsilon)}} |\xi|^{s-1}d\xi= \frac{2}{s}\cdot q^{sk(\tau+\epsilon)}.
\]
It is easy to verify that $\widetilde{\varphi}_n(x,y) =2 a_n\sin \pi\lambda_n(x-y) $  satisfies the condition \eqref{eq3.4} with $\delta=1/2$.  So \eqref{eq3.5} in Lemma \ref{lemma3.3} also holds for $\widetilde{\varphi}_n(x,y)$. Similar to Theorem \ref{thm3.4}, we have if $s<\ell/2$,
\begin{align*}
|\widetilde{U}_k|&\leq\int_{|\xi|>q^{k(\tau+\epsilon)}}|\xi|^{s-1}\prod_{j=k-\ell}^{k-1}\prod_{q^j\leq \lambda_n<q^{j+1}}|J_0(4\pi \xi a_n \sin\pi (\lambda_n x-\lambda_ny))|d\xi\lesssim_{\epsilon} q^{sk(\tau+\epsilon)}.
\end{align*}
The estimates of $\widetilde{B}_k$ and $\widetilde{U}_k$ is combined to imply that \eqref{eq5.3} holds.
\end{proof}
\begin{remark}
Recall that
\[
\mathbb{E}[|\widehat{\mu_2}(\xi)|^2]=\iint_{A\times A} \prod_{n=1}^{\infty} \Bigl(\int_0^1e^{-2\pi i\xi a_n(\sin2\pi(\lambda_nx+\theta)-\sin2\pi(\lambda_ny+\theta))}d\theta\Bigr) d\nu(x)d\nu(y).
\]
If $(x,y)\in A_k$, $k-\ell\leq j<k-1$ and $q^{j}\leq \lambda_n<q^{j+1}$, we see that
\[
0<\frac{1}{2q^{\ell+1}}\leq \lambda^n|x-y|\leq \frac{1}{2}.
\]
The key to proving the above theorem lies in estimating the following oscillatory integral:
\[
\Bigl|\int_0^1e^{-2\pi i t(\sin2\pi(x+\theta)-\sin2\pi\theta)}d\theta\Bigl|=|J_0(4\pi t\sin\pi x)|\leq (1+c|t|^{2})^{-\frac{1}{4}},\quad \frac{1}{2q^{\ell+1}}\leq |x|\leq \frac{1}{2}.
\]
For Weierstrass-type functions ($1<\inf \lambda_{n+1}/\lambda_n\leq \sup\lambda_{n=1}/\lambda_n<\infty$), polynomial decay $C|t|^{-\gamma}$ is sufficient.
\end{remark}

Although Hunt \cite{Hunt} already obtained the Hausdorff dimension of the graph of the random one-dimensional  Weierstrass function over $[0, 1]$,  we can use the methods in Section \ref{sec4} to estimate the  Hausdorff  dimension of the graph of $S_1(x)$ over an arbitrary Borel set. The proofs are essentially the same as those in Section \ref{sec4}.  First, if $x\neq y$, the characteristic function of the one-dimensional random variable $S_1(x)-S_1(y)$ is
\[
\psi_{x,y}(\xi)=\mathbb{E}\bigl[e^{-2\pi i\xi(S_1(x)-S_1(y))}\bigr]=\prod_{n=1}^{\infty}J_0(4\pi a_n\xi\sin\pi\lambda_n(x-y))
\]
Similar to Proposition \ref{prop4.2}, we have an upper bound on the density function of $S_1(x)-S_1(y)$.
\begin{proposition}\label{prop5.6}
Let $0<|x-y|\leq 1/(2q^3)$. If $0\leq \tau<\infty$, then  for any sufficiently small $\epsilon>0$,
\[
|f_{x,y}(u)|\lesssim_{\epsilon} |x-y|^{-(\tau+\epsilon)},\quad u\in\mathbb{R}.
\]
\end{proposition}
\begin{proof}
Take $\ell=3$ in Lemma \ref{lemma3.3}. Let $k\geq 3$ such that $1/(2q^{k+1})<|x-y|\leq 1/(2q^{k})$. Similar to Proposition \ref{prop4.2},  we have
\begin{align*}
\int_{\mathbb{R}}|\psi_{x,y}(\xi)|d\xi&\lesssim \int_{0}^{q^{k(\tau+\epsilon)}}d\xi+ \int_{q^{k(\tau+\epsilon)}}^{\infty} \prod_{j=k-3}^{k-1}\prod_{q^{j}\leq \lambda_n<q^{j+1}}\Bigl|J_0(4\pi \xi a_n \sin\pi\lambda_n(x-y)\Bigr|d\xi\\
& \lesssim_{\epsilon}q^{k(\tau+\epsilon)} + q^{\frac{3k(\tau+\epsilon)}{2}}\int_{q^{k(\tau+\epsilon)}}^{\infty}   \xi^{-\frac{3}{2}} d\xi\lesssim_{\epsilon}|x-y|^{-(\tau+\epsilon)}.
\end{align*}
The upper bound of $|f_{x,y}(u)|$ follows from  \eqref{eq4.1}.
\end{proof}

Using this proposition, we can estimate the lower bound of the Hausdorff  dimension of   $G_{S_1}(A)$.

\begin{theorem}\label{thm5.7}
If $0< \sigma\leq \tau\leq 1$, for any Borel set $A\subset\mathbb{R}$, then (i) almost surely,
\[
\min\Bigl\{ \frac{\hd A}{\tau},\; \hd A+1-\tau\Bigr\}\leq \hd G_{S_1}( A)\leq\min\Bigl\{\frac{\hd A}{\sigma},\; \hd A+1-\sigma\Bigr\}.
\]
(ii) If $\tau=0$ and $\hd A>0$, then $\hd G_{S_1}( A)=1+\hd A$ a.s.
\end{theorem}
\begin{proof}
\textbf{Upper bound.} Since the sine  series  $S_1(x)$ in \eqref{eq5.2} is the  imaginary part of the complex series \eqref{eq1.8}, it follows
from Theorem \ref{thm2.3} that the sine series $S_1(x)$ is also almost surely H\"{o}lder continuous with order $\min\{\sigma,1\}-\epsilon$ for any sufficiently small $\epsilon>0$. The upper bound follows from  10.6 of \cite{Kahane}.

\textbf{Lower bound.} Case 1: $\tau>\hd A$. By the  projection and Theorem \ref{thm5.5},
\[
\hd G_{S_1}(A)\geq \hd S_1(A)\geq \frac{\hd A}{\tau}.
\]

Case 2: $0\leq \tau<\hd A$. Assume that   $\mathrm{diam}\,A \leq 1/(2q^3)$, we choose $\alpha\nearrow\hd A$ and $\epsilon\searrow 0$ such that
\[
\tau+\epsilon<\alpha.
\]
Let  $\nu$ be a Borel measure supported on $A$ satisfying \eqref{eq3.9}, and let  $\mu_3$ be  the push-forward measure supported on $G_{S_1}(A)$:
\[
\mu_3(E)=\nu(\{x\in A\colon G_{S_1}(x)\in E\}),\quad \forall\, \mathrm{Borel}\, E\subset\mathbb{R}^2.
\]
Let $A_k=\{(x,y)\in A\times A: 1/(2q^{k+1})<|x-y|\leq 1/(2q^k)\}$. It follows from Proposition \ref{prop5.6} that if $1<s<1+\alpha-\tau-\epsilon$,
\begin{align*}
\mathbb{E}[I_s(\mu_3)]&=\iint \mathbb{E}\bigl[(|x-y|^2+|S_1(x)-S_1(y)|^2)^{-\frac{s}{2}}\bigr]d\nu(x)d\nu(y)\\
&=\iint \int_{\mathbb{R}}\frac{f_{x,y}(u)}{(|x-y|^2+u^2)^{\frac{s}{2}}}du d\nu(x)d\nu(y)\lesssim_{\epsilon} \iint \int_{0}^{\infty}\frac{|x-y|^{-(\tau+\epsilon)}}{(|x-y|^2+u^2)^{\frac{s}{2}}}du d\nu(x)d\nu(y)\\
&\lesssim_{\epsilon}\sum_{k=3}^{\infty}\iint_{A_k}|x-y|^{1-s-\tau-\epsilon}d\nu(x)d\nu(y)\int_0^{\infty}\frac{du}{(1+u^2)^{\frac{s}{2}}}\lesssim_{\epsilon} \sum_{k=3}^{\infty}q^{-k(1-\tau-\epsilon-s+\alpha)}<\infty.
\end{align*}
This implies that
\[
\hd G_{S_1}(A)\geq \hd A+1-\tau.
\]

Case 3: $\tau=0$ and $\hd A>0$. The lower bound comes from Case 2, and the upper bound follows from $G_{S_1}(A)\subset A\times \mathbb{R}$.
\end{proof}


We finally apply  the theorems from this section to the random real-valued Riemann function:
\[
\widetilde{R}_{a,b,\Theta}(x) =\sum_{n=1}^{\infty}\frac{\sin2\pi (n^a x+\theta_n)}{n^b},\quad x\in\mathbb{R},
\]
where $\theta_n\overset{\mathrm{iid}}{\sim} U[0,1]$. In particular, $\hd(\widetilde{R}_{2,2,\Theta})([0,1]) = 5/4$ almost surely, which is also the box dimension result where Chamizo and C\'{o}rdoba obtained in \cite{CC1999}.

\begin{corollary}
If $1<b\leq a+\frac{1}{2}$, for any Borel set $A\subset\mathbb{R}$, then (i) almost surely,
\begin{align*}
\hd \widetilde{R}_{a,b,\Theta}(A)&= \min\Bigl\{1,\; \frac{2a}{2b-1}\hd A\Bigr\},\\
\hd G_{\widetilde{R}_{a,b,\Theta}}(A)&= \min\Bigl\{\frac{2a}{2b-1}\hd A,\;\hd A+1-\frac{2b-1}{2a}\Bigr\}.
\end{align*}
(ii) If $\hd A>\frac{2b-1}{2a}$,  $\widetilde{R}_{a,b,\Theta}(A)$ has positive  one-dimensional Lebesgue measure a.s.  (iii) If $\hd A>\frac{2b-1}{a}$, $\widetilde{R}_{a,b,\Theta}(A)$ has interior points a.s. (iv) If $\bd A=\hd A$, then the above two dimension formulas also hold for the box-counting dimension.
\end{corollary}
Similar result  can be obtained for the random real-valued Weierstrass-type functions, but there are too many  dimension results of the graph of the real-valued Weierstrass-type functions, so we omit it here.

\section{Open Questions}\label{sec6}
In this section, we  are going to pose some open questions and conjectures. We first consider the conjectures in the deterministic case. Then will talk pose some interesting problems in the random scenario.

\subsection{Deterministic case}
Consider the following deterministic complex Weierstrass function:
\[
W_{\beta, \lambda}(x)=\sum_{n=1}^\infty  \lambda^{-\beta n} e^{2\pi i\lambda^nx},\quad 0<\beta\leq 1,\; \lambda>1.
\]
Drawing upon all existing results \cite{Belov, SZ, KWW,RS2021,RS2024}, our observations (Figure  \ref{fig1}), and Corollary \ref{coro1.3}, we hereby propose the following conjecture for $W_{\beta, \lambda}$.

\begin{conjecture}
For all $\lambda>1$ and $0<\beta\leq 1$, then (i) for all Borel $A\subset\mathbb{R}$,
\begin{align*}
\hd W_{\beta, \lambda}(A)&=\min\left\{2,\; \frac{\hd A}{\beta}\right\},\\
\hd G_{W_{\beta, \lambda}}( A)&=\min\Bigl\{\frac{\hd A}{\beta},\; \hd A+2-2\beta\Bigr\}.
\end{align*}
(ii) If $0<\beta<\frac{1}{2}\hd A$, then $W_{\beta, \lambda}(A)$ has  interior points. (iii) If $\bd A=\hd A$, then the above two dimension formulas also hold for the box-counting dimension.
\end{conjecture}
Similarly, we have the conjecture for the deterministic complex and real-valued  Riemann functions:
\[
R_{a, b}(x)=\sum_{n=1}^{\infty}  \frac{e^{2\pi i n^{a}x}}{n^b},\quad \widetilde{R}_{a, b}(x)=\sum_{n=1}^{\infty} \frac{\sin2\pi n^ax}{n^b},\quad a>0,\;\; b>1.
\]
\begin{conjecture}
If $1<b\leq a+\frac{1}{2}$, for any Borel set $A\subset\mathbb{R}$, then (i)
\begin{align*}
\hd R_{a,b}(A)&= \min\Bigl\{2,\; \frac{2a}{2b-1}\hd A\Bigr\},\\
\hd G_{R_{a,b}}(A)&= \min\Bigl\{\frac{2a}{2b-1}\hd A,\;\hd A+2-\frac{2b-1}{a}\Bigr\},\\
\hd G_{\widetilde{R}_{a,b}}(A)&= \min\Bigl\{\frac{2a}{2b-1}\hd A,\;\hd A+1-\frac{2b-1}{2a}\Bigr\}.
\end{align*}
(ii) If $\hd A>\frac{2b-1}{a}$,  $R_{a,b}(A)$ has interior points. (iii) If $\bd A=\hd A$, then the above three dimension formulas also hold for the box-counting dimension.
\end{conjecture}

\subsection{Random case}
Next, we consider the random complex Weierstrass function in \eqref{eq1.9}:
\[
W_{\beta}(x)=W_{\beta,\lambda,\Theta}(x)=\sum_{n=1}^{\infty}\lambda^{-\beta n}e^{2\pi i(\lambda^nx+\theta_n)},\quad \lambda>1,\; \beta\in (0, 1], \;\; \theta_n\overset{\mathrm{iid}}{\sim} U[0,1],\;\; x\in\mathbb{R}.
\]
\subsubsection{Interior points}
It is noteworthy  that Proposition \ref{prop3.1} relies on the singularity of the measure supported on $E$. It is a direct consequence of whether the Fourier transform of measure belongs to $L^2$ or $L^1$ \cite[Theorems 3.3 and 3.4]{Mattila}; one may refer to \cite{GH} for the random versions. Let $\mu={W_{\beta}}_{\ast}\nu$ be the push-forward measure given by \eqref{eq3.10} and $\mathcal{L}_2$ be the two-dimensional Lebesgue measure. According to Theorem \ref{thm1.1} and Proposition \ref{prop3.1}, the singularity of $\mu$ is as follows:
\begin{corollary}
(i) If $\beta<\frac{1}{2}\hd A$, then $\mu\ll\mathcal{L}_2$ a.s.  (ii) If $\beta<\frac{1}{4}\hd A$,then $\frac{d\mu}{d\mathcal{L}_2}$ is continuous about $\xi$ a.s. (iii) If $\beta>\frac{1}{2}\hd A$, then $\mu\perp\mathcal{L}_2$, a.s.
\end{corollary}
The  statement (iii) comes from the H\"{o}lder property. From a probability perspective, the existence of interior points in $W_{\beta}(A)$  is related to the continuity of the local time of $W_{\beta}$ on $A$, where the local time $l(\xi)$ is defined to be the Radon-Nikodym derivative $\frac{d\mu}{d\mathcal{L}_2}(\xi)$ if $\mu\ll\mathcal{L}_2$.
For a stochastic process $X_t$, local time is a useful tool to study fractal properties of the level set $X^{-1}(x)=\{t: X(t)=x\}$. For more information, please refer to \cite{GH,SX,Xiao}.

 Belov  \cite{Belov} proved that for any $\beta<1/2$, sufficiently large $\lambda$ and any deterministic  $\Theta$, the  deterministic Weierstrass function $W_{\beta,\lambda,\Theta}$ is space-filling on all large intervals. We hope that this result holds in general, which requires proving that $\widehat{\mu}\in L^1(d \mathbb{P}\times d\mathcal{L}_2)$ when $0<\beta<\frac{1}{2}\hd A$. We cannot do this by simply applying Proposition \ref{prop3.1} (ii). Here is the open question:

\begin{question}
 If $\frac{1}{4}\hd A\leq \beta<\frac{1}{2}\hd A$, is it possible that  $\widehat{\mu}\in L^1(d \mathbb{P}\times d\mathcal{L}_2)$, or  $W_{\beta}(A)$ has interior points?
\end{question}

\subsubsection{Uniform dimension result}
In Corollary \ref{coro1.3}, we see that the exceptional even of  probability zero, on which \eqref{eq1.10} fails, depends on the Borel $A$. So, the Borel set in Corollary \ref{coro1.3}  can not be random.  A natural question is whether there  exists a null probability event $O$ such that for every $\Theta\notin O$, \eqref{eq1.10} holds for all Borel sets. Such a result, if exists, is called a uniform Hausdorff  dimension result and is applicable even if $A$ is a random set. The first uniform dimension result for Brownian motion is due to Kaufman \cite{Kau}.  For the planar fractional Brownian motion $B_{\beta}(t)$ with Hurst index $\beta\in (0,1)$, Monrad and Pitt \cite{MP} proved the following uniform dimension result  (we only introduce the planar case for comparison with our results): if $\beta\geq 1/2$, then almost surely,
 \begin{equation}\label{eq6.1}
\hd B_{\beta} (A) =\frac{1}{\beta}\hd A,\quad \text{for all closed} \; A\subset\mathbb{R}.
 \end{equation}
 Moreover, they also proved that if $\beta<1/2$, then $\hd B_{\beta}^{-1}(F)=1-2\beta+\beta\hd F$ for  every closed set $F\subset\mathbb{R}^2$,  almost surely.

It follows from Proposition \ref{prop5.1} that $W_{\beta}$ is $\beta$-H\"{o}lder for $\beta\in(0, 1), \lambda>1$ and all $\Theta$. Similarly, we can prove that $W_{1}$ is $(1-\epsilon)$-H\"{o}lder for any $\epsilon\in(0, 1)$ and all $\Theta$. Therefore, for all $\beta\in [1/2, 1], \lambda>1$ and all $\Theta$, we have
\[
\hd W_{\beta}(A)\leq \frac{\hd A}{\beta}, \quad \forall\, \text{Borel}\; A\subset \mathbb{R}.
\]
Note that $A$ here can be related to $\Theta$. We want to know whether the lower bound also holds almost surely.

\begin{question}
Does the uniform dimension result \eqref{eq6.1} hold for the random Weierstrass function $W_{\beta}$ for $\beta\in [1/2, 1]$? If not, what is the Hausdorff dimension of the preimage $W_{\beta}^{-1}(F)$ for  closed set $F\subset\mathbb{R}^2$?
\end{question}

\subsubsection{Cantor boundary behavior}
We now regard the random Weierstrass function in \eqref{eq1.9} as  the boundary value of an analytic function within the unit circle (for simplicity, we assume $\lambda\geq 2$ is an integer):
\[
W_{\Theta}(z)=W_{\beta, \lambda, \Theta}(z)=\sum_{n=1}^{\infty}(\lambda^{-\beta n}e^{2\pi i\theta_n})\cdot z^{\lambda^n},\quad z\in \overline{\mathbb{D}}.
\]
Let $\Gamma_{\beta, \lambda, \Theta}$ be the unique unbounded connected component of $\widehat{\mathbb{C}}\setminus W_{\beta, \lambda, \Theta}(\partial\mathbb{D})$, we call $\partial\Gamma_{\beta, \lambda, \Theta}$  the outer boundary of $W_{\beta, \lambda, \Theta}(z)$. Dong, Lau and Liu \cite{DLL} showed that  that $\mathcal{C}_{\beta,\lambda, \Theta}=W_{\Theta}^{-1}(\partial\Gamma_{\beta, \lambda, \Theta})\subset \partial\mathbb{D}$ is a Cantor-type set for  all $\beta\in(0, 1), \lambda\geq 2$ and every sample path $\Theta$ (by a Cantor-type set, we mean the set is compact, perfect and totally disconnected).

Studying the Hausdorff dimension of the outer boundary has been challenging problem and yet it connects to other areas. Lawler, Schramm, and Werner developed the SLE theory and used it to prove that the Hausdorff dimension of the outer boundary of the planar Brownian motion is $4/3$ almost surely  \cite{Lawler}. Lind and Robins \cite{LR} studied the deterministic Loewner equation driven by some real-valued Weierstrass function and showed that it exhibits a phase transition as in SLE.

We have not found a viable approach to compute the dimensions of the Cantor-type set and the outer boundary of the Weierstrass function; readers may refer to \cite{DLL} for further details. For the random model, note that
\[
W_{\Theta}(\mathcal{C}_{\beta,\lambda, \Theta})=\partial\Gamma_{\beta,\lambda,\Theta}.
\]
If the uniform dimension result holds for $W_{\beta, \lambda, \Theta}$ with some $\beta\in[1/2, 1)$ and $\lambda\geq 2$, then almost surely,
\[
\hd \partial \Gamma_{\beta,\lambda,\Theta} =\frac{1}{\beta} \hd \mathcal{C}_{\beta,\lambda,\Theta},
\]
and hence $ \hd \mathcal{C}_{\beta,\lambda,\Theta}\geq \beta$  (this is because $\partial \Gamma_{\beta,\lambda,\Theta}$ is locally connected \cite{DLL}, so its Hausdorff  dimension is at least 1).
\begin{question}
What are the Hausdorff dimensions of the random Cantor-type set $\mathcal{C}_{\beta,\lambda,\Theta}$ and the outer boundary  $\partial \Gamma_{\beta,\lambda,\Theta}$? What about the same question in the deterministic case?
\end{question}

\end{document}